\documentclass[11pt,english,oneside]{amsart}

\usepackage[breaklinks,colorlinks,linkcolor=blue,citecolor=blue,urlcolor=black]{hyperref}
\usepackage[T1]{fontenc}
\allowdisplaybreaks[4]
\usepackage{geometry}
\usepackage{indentfirst}
\usepackage{amssymb}
\usepackage{color}
\usepackage{graphicx}
\usepackage{subfigure}
\usepackage{enumerate}
\usepackage{float}

\usepackage{booktabs}
\usepackage{array, caption, threeparttable}
\usepackage[font=small,labelfont=bf,labelsep=none]{caption}
\usepackage{graphicx,color}
\usepackage{amsmath, amssymb, graphics}
\usepackage{graphicx}
\usepackage{epstopdf}

\numberwithin{figure}{section}

\makeatletter
\numberwithin{equation}{section} %% Comment out for sequentially-numbered
\numberwithin{figure}{section} %% Comment out for sequentially-numbered
\@ifundefined{theoremstyle}{\usepackage{amsthm}}{}
\theoremstyle{plain}
\newtheorem{theorem}{Theorem}[section]

\newtheorem{proposition}[theorem]{Proposition}

\newtheorem{lemma}[theorem]{Lemma}
\newtheorem{definition}[theorem]{Definition}

\usepackage{geometry}

\newcommand{\R}{\mathbb{R}}

\newcommand{\Z}{\mathbb{Z}}

\def\e{\hbox{\rm e}}

\smallskip
\def\<{{\langle }}
\def\>{{\rangle }}

\makeatother

\usepackage{babel}

\def\e{\hbox{\rm e}}

\smallskip
\def\<{{\langle }}
\def\>{{\rangle }}

\makeatother

\usepackage[mathscr]{eucal}
\usepackage{amssymb}
\usepackage{latexsym}
\usepackage{amsthm}

\makeatletter
\@addtoreset{equation}{section}

\usepackage{color}

\usepackage{graphicx,color}

\title[Construction of compact embedded $\lambda$-hypersurfaces via isoparametric hypersurfaces]
{Construction of compact embedded $\lambda$-hypersurfaces via isoparametric hypersurfaces}
\author{Junqi Lai and  Guoxin Wei}
\address{Junqi Lai \\  School of Mathematical Sciences, South China Normal University,
510631, Guangzhou,  China, 2019021668@m.scnu.edu.cn}
\address{Guoxin Wei \\  School of Mathematical Sciences, South China Normal University,
510631, Guangzhou,  China, weiguoxin@tsinghua.org.cn}

\begin{document}

	\maketitle
	
	\begin{abstract} 
		We construct closed embedded \(\lambda\)-hypersurfaces in \(\mathbb{R}^{n+1}\) of topological type \(S^1\times M\) for any \(\lambda\ge 0\), where \(M\subset\mathbb{S}^n\) is an isoparametric hypersurface with equal principal curvature multiplicities. This extends previous results of Angenent, McGrath, and Riedler (for \(\lambda=0\)) as well as Cheng–Wei and Ross (for \(\lambda>0\)), and in particular completes the missing case of Riedler’s construction.
	\end{abstract}
	
	\footnotetext{ }
	
	\section{Introduction}
	
	A hypersurface \(\Sigma^n \subset \R^{n+1}\) is called a \(\lambda\)-hypersurface if it satisfies
	\begin{equation}\label{eq:2-27-1}
		H+\langle X,\nu\rangle=\lambda,
	\end{equation}
	where \(\lambda\) is a constant, \(X\) is the position vector, \(\nu\) is a unit normal vector and \(H\) is the mean curvature.
	If $\lambda=0$, then $X:\Sigma^n\to  \mathbb{R}^{n+1}$ is a self-shrinker of  mean curvature flow, which plays an important role for study on singularities of the mean curvature flow.

	The notation of $\lambda$-hypersurfaces were first introduced by Cheng and Wei in \cite{CW} (see also \cite{MR}). Cheng and Wei \cite{CW} proved that $\lambda$-hypersurfaces are critical points of the weighted area functional with respect to weighted
	volume-preserving variations. The equation of $\lambda$-hypersurfaces also arises in the study of isoperimetric problems in Gaussian Euclidean spaces, which is a long-standing topic studied in various fields in science. $\lambda$-hypersurfaces can also be viewed as stationary solutions to the isoperimetric problem in the Gaussian space.

	For \(\lambda=0\), Angenent~\cite{A} constructed an \(S^1\times S^{n-1}\) embedding of self-shrinker in \(\mathbb{R}^{n+1}\). 
	McGrath~\cite{M} constructed an \(S^1\times S^{n}\times S^n\) embedding of self-shrinker in \(\mathbb{R}^{2n+2}\). 
	Recently, Riedler~\cite{R1} constructed an \(S^1\times M\) embedding of self-shrinker in \(\mathbb{R}^{n+1}\) for an isoparametric hypersurface \(M\) in \(\mathbb{S}^n\) whose multiplicities of principal curvatures are equal, thereby generalizing the results of Angenent and McGrath.
	It is natural to ask whether these results can be generalized to all \(\lambda\in \mathbb{R}\).
	Cheng and Wei~\cite{CW1} extended Angenent's results to \(\lambda>0\); Ross~\cite{R} extended McGrath's results to the same case. Riedler's result, however, remains to be extended.
	In this paper, we successfully provide such an extension.

	Following the approach of \cite{R1}, we can prove:
	\begin{theorem}\label{thm:2-27-1}
		Let \(\lambda \ge0\), then for any isoparametric hypersurface \( M \) in \( \mathbb{S}^n \), \( n \geq 2 \), for which the multiplicities \( m_1 \) and \( m_2 \) of the principal curvatures agree, there is a closed embedded \(\lambda\)-hypersurface of topological type \( S^1 \times M \) in \( \mathbb{R}^{n+1} \). This hypersurface is a union of homothetic copies of the leaves of the isoparametric foliation of \( \mathbb{S}^n \) associated to \( M \).
	\end{theorem}
	% \begin{remark}
	% 	This theorem can be viewed as a generalization of theorem 1.1 in \cite{CW1} and theorem 1 in \cite{R} since the parallel \((n-1)\)-spheres in \(\mathbb{S}^n\) are isoparametric with one distinct principal curvature and the Clifford tori \(\mathbb{S}^{n/2}(\rho)\times \mathbb{S}^{n/2}(\sqrt{1-\rho^2})\) in \(\mathbb{S}^n\) are isoparametric with two distinct principal curvatures.
	% \end{remark}

%%%%%%%%%%%%%%%%%%%%%%%%%%%%%%%%%%%%%%%%%%%%%%%%%%%%%%%%%%%%%%%%%%%%%%%%%%%%%%%%%%%%%%%%%%%%%%%%%%%%%%%%%%%%%%%%%%%%%%%%%%%%%%%%%%%%%%%%%%%%%%%%%%%%%%%%%%%%%%%%%%%%%%%%%%%%%%%%%%

	\section{Preliminaries}\label{prel}
	Let \(M^{n-1}\) be a isoparametric hypersurface in the unit sphere \(\mathbb{S}^n \subset\R^{n+1}\),
	\(N\) be a unit normal vector field of \(M\).
	It is well known that the principal curvatures of \(M\) are constant.
	Let \(g\) be the number of distinct principal curvatures of \(M\)
	and \(\cot \varphi_1  \ge \cot \varphi_2 \ge\dots \ge \cot \varphi_{n-1}\) be principal curvatures of \(M\), where \(\varphi_1,\dots ,\varphi_{n-1} \in (0,\pi)\).
	Let \(I\) be an interval and \((x(s),y(s))=(r(s)\cos\varphi(s),r(s)\sin\varphi(s)),\ s \in I\) be a smooth arc-length curve, where \((r(s),\varphi(s))\in(0,\infty)\times(\varphi_1-\frac{\pi}{g},\varphi_1)\).
	Define an immersion \(F\) from \(M\times I\) into \(\R^{n+1}\) by
	\[
	F(p,s)=p\,x(s)+N(p)\,y(s),\ \ \forall (p,s) \in M\times I.
	\]

	One readily checks that 
	\[
	\nu(p,s)=-p\,y'(s)+N(p)\, x'(s)
	\]
	is a unit normal vector field of \(F\) and the principal curvatures of \(F\) are 
	\[
	\begin{aligned}
		&\kappa_i=\frac{x'\cot \varphi_i+y'}{x-y\cot \varphi_i}=\frac{r'}{r}\cot(\varphi_i-\varphi)+\varphi',\ \ i=1,\dots,n-1,\\
		&\kappa_n=x'y''-y'x''.\\
	\end{aligned}
	\]
	Let \(\vartheta\) denote the angles between tangent vectors of the curve \((x(s),y(s))\) and \(x\)-axis and \(\alpha=\vartheta-\varphi\), then
	\begin{equation}\label{eq:2-27-2}
		r'=\cos \alpha,\ \  \varphi'=\frac{1}{r}\sin \alpha
	\end{equation} 
	and 
	\[
	\kappa_n=\vartheta'=\alpha'+\varphi'.
	\]
	 
	To proceed, let us note that all distinct principal curvatures of \(M\) are of the form \( \cot \varphi_1,\,\cot(\varphi_1+\frac{\pi}{g}),\,\dots,\cot(\varphi_1+\frac{(g-1)\pi}{g})\) and their multiplicities repeat every two consecutive terms in the sequence, and further, all multiplicities are equal when \( g \) is odd. These fundamental results are proved on page~108 of \cite{Ce}.
	Let \(m_2\) denotes the multiplicities of the largest principal curvature of \(M\) and \(m_1\) denotes that of the second largest (if \(g=1\) then let \(m_1=m_2\)), then the mean curvature of \(F\) is 
	\[
    H=\frac{r'}{r}\frac{g }{2}\left( m_2 \cot\left( \frac{g}{2}(\varphi_1-\varphi) \right) -m_1 \tan\left( \frac{g}{2}(\varphi_1-\varphi) \right)\right)+n\,\varphi'+\alpha'.
	\]

	If \(F\) is a \(\lambda\)-hypersurface, then one concludes from \eqref{eq:2-27-1} that 
	\begin{equation}\label{eq:2-27-3}
		\alpha'=\left(r-\frac{n}{r}\right)\sin \alpha -\frac{\cos \alpha}{r}\frac{g }{2}\left( m_2 \cot\left( \frac{g}{2}(\varphi_1-\varphi) \right) -m_1 \tan\left( \frac{g}{2}(\varphi_1-\varphi) \right)\right)+\lambda.
	\end{equation}
	Let \(\phi=\frac{\pi}{g}-\varphi_1+\varphi\), then from \eqref{eq:2-27-2} and \eqref{eq:2-27-3} one sees that \(r,\ \phi,\ \alpha\) satisfy the following system:
	\[
		\left\lbrace
    	\begin{aligned}
    	r'&= \cos \alpha, \\
    	\phi'&=\frac{1}{r}\sin \alpha, \\
    	\alpha'&=\sin\alpha\left(r-\frac{n}{r}\right)+\frac{g}{2}\frac{\cos\alpha}{r}\left(m_1\cot\left(\frac{g}{2}\phi\right)-m_2\tan\left(\frac{g}{2}\phi\right)\right)+\lambda.
    	\end{aligned}
    	\right.
	\]
	Up to a reparametrization of the argument, this system is equivalent to the following:
	\begin{equation*}
    	\left\lbrace
    	\begin{aligned}
    	r'&= r\cos \alpha, \\
    	\phi'&=\sin \alpha, \\
    	\alpha'&=\sin\alpha(r^2 -n)+\frac{g}{2}\cos\alpha\left(m_1\cot\left(\frac{g}{2}\phi\right)-m_2\tan\left(\frac{g}{2}\phi\right)\right)+\lambda\,r.
    	\end{aligned}
    	\right.
    \end{equation*}
    Let $\theta=\frac{g}{2}\phi$, $\xi=\frac{g}{2}\ln\left(\sqrt{\frac{2}{g}}r \right) $, then $r=\sqrt{\frac{g}{2}}\,\e^{\frac{2}{g}\xi}$ and 
    \begin{equation*}
    \left\lbrace
    \begin{aligned}
    \xi'&= \frac{g}{2}\cos\alpha, \\
    \theta'&= \frac{g}{2}\sin\alpha, \\
    \alpha'&=\frac{g}{2}\left[ \sin\alpha\,(\e^{\frac{4}{g}\xi} -m) + \cos\alpha\,(m_1\cot\theta - m_2\tan\theta)+\sqrt{\frac{2}{g}}\,\lambda\,\e^{\frac{2}{g}\xi}\right] ,
    \end{aligned}
    \right.
    \end{equation*}
	
    where $m=\frac{2}{g}\,n$.
    As before, this system is equivalent to the following:
     \begin{equation}\label{eq:10-25-0}
    \left\lbrace
    \begin{aligned}
    \xi'&= \sin 2\theta\cos\alpha, \\
    \theta'&= \sin 2\theta\sin\alpha, \\
     \alpha'&= \sin 2\theta\sin\alpha\,(\e^{\frac{4}{g}\xi} -m) + 2\cos\alpha\,(m_1\cos^2\theta - m_2\sin^2\theta)+\sqrt{\frac{2}{g}}\,\lambda\,\sin 2\theta\,\e^{\frac{2}{g}\xi}. 
    \end{aligned}
    \right.
    \end{equation}

	\noindent
	When ${\theta}'>0$ the profile curve $(\xi(t),\theta(t))$ can be written in the form $(\xi(\theta), \theta)$, by (\ref{eq:10-25-0}), the function $\xi(\theta)$ satisfies the differential equation
	\begin{equation}\label{eq:10-25-1}
	\begin{aligned}
		\frac{\mathrm d^2\xi}{\mathrm d\theta^2} &= -\left( 1 + \left( \frac{\mathrm d\xi}{\mathrm d\theta}\right) ^2\right) \Bigg( \e^{\frac{4}{g}\xi} -m + (m_1\cot\theta - m_2\tan\theta)\frac{\mathrm d\xi}{\mathrm d\theta}  \\
	& \  \ \ \ \ \ +\lambda\sqrt{1 + \left( \frac{\mathrm d\xi}{\mathrm d\theta}\right) ^2}\sqrt{\frac{2}{g}}\,\e^{\frac{2}{g}\xi}\Bigg).
	\end{aligned}
	\end{equation}
	When ${\theta}'<0$ the profile curve $(\xi(t),\theta(t))$ can also be written in the form $(\xi(\theta), \theta)$, by (\ref{eq:10-25-0}), the function $\xi(\theta)$ in this situation satisfies the differential equation
	\begin{equation}\label{eq:10-25-2}
	\begin{aligned}
	\frac{\mathrm d^2\xi}{\mathrm d\theta^2} &= -\left( 1 + \left( \frac{\mathrm d\xi}{\mathrm d\theta}\right) ^2\right) \Bigg( \e^{\frac{4}{g}\xi} -m + (m_1\cot\theta - m_2\tan\theta)\frac{\mathrm d\xi}{\mathrm d\theta}  \\
	& \  \ \ \ \ \ -\lambda\sqrt{1 + \left( \frac{\mathrm d\xi}{\mathrm d\theta}\right) ^2}\sqrt{\frac{2}{g}}\,\e^{\frac{2}{g}\xi}\Bigg).
	\end{aligned}
	\end{equation}
	Let $ D:= \R\times(0,\frac{\pi}{2})\times \R$, we are only interested in those solutions of \eqref{eq:10-25-0} in ${D}$.
	Let us conclude this section by some elementary properties of solutions of \eqref{eq:10-25-0}, proofs are standard and are thus omitted.
	\begin{proposition}\label{prop:10-25-3}
	For solutions of \eqref{eq:10-25-0}, we have:
	\begin{enumerate}[{\rm (i)}]
		\item For any $(\xi_0,\theta_0,\alpha_0) \in \R^3$, there is a unique solutions $\gamma$ of \eqref{eq:10-25-0} with initial condition $\gamma(0)=(\xi_0,\theta_0,\alpha_0)$. 
		This solutions is smooth and has the domain of definition all of $\R$, i.e. solutions exists for all times.
		
		\item Suppose $(\xi_s,\theta_s,\alpha_s) $ converges in $\R^3$ to a point $(\xi_\infty,\theta_\infty,\alpha_\infty) $. 
		Denote by $\gamma_s$ the solutions of \eqref{eq:10-25-0} with initial condition $(\xi_s,\theta_s,\alpha_s)$ and $\gamma_\infty$ the solution of \eqref{eq:10-25-0} with initial condition $(\xi_\infty,\theta_\infty,\alpha_\infty)$.
		Then $\gamma_s$ converges uniformly on compacta to $\gamma_\infty$.
		
		\item Solutions of \eqref{eq:10-25-0} with initial condition in $D$ remains in $D$ for all times.
	\end{enumerate}
	\end{proposition}
	\begin{proposition}\label{prop:2-27-1}
	Let \(m_1=m_2\)	and \((\xi(t),\theta(t),\alpha(t))\) be a solution of \eqref{eq:10-25-0} in \(D\), then \(\gamma(t)=(\xi(-t),\frac{\pi}{2}-\theta(-t),\pi-\alpha(-t))\) is also a solution of \eqref{eq:10-25-0} in \(D\).
	\end{proposition}

%%%%%%%%%%%%%%%%%%%%%%%%%%%%%%%%%%%%%%%%%%%%%%%%%%%%%%%%%%%%%%%%%%%%%%%%%%%%%%%%%%%%%%%%%%%%%%%%%%%%%%%%%%%%%%%%%%%%%%%%%%%%%%%%%%%%%%%%%%%%%%%%%%%%%%%%%%%%%%%%%%%%%%%%%%%%%%%%%%

\section{Existence of Periodic Curves}
Throughout this paper \(\arctan(\sqrt{m_1/m_2})\) will be denoted by \(\theta^*\). 
To prove theorem \ref{thm:2-27-1}, it suffices to find solutions of \eqref{eq:10-25-0} such that the components \((\xi(t),\theta(t))\) form smooth simple closed curves.
Let us begin by a definition.
\begin{definition}\label{def:10-25-4}
	For $\xi_0 \in \R$ let $(\xi_{\xi_0}(t),\theta_{\xi_0}(t),\alpha_{\xi_0}(t))$ denote the solution of \eqref{eq:10-25-0} with initial condition $(\xi(0),\theta(0),\alpha(0)=(\xi_0, \theta^*,\frac{\pi}{2})$. Then:
	\begin{enumerate}[{\rm (i)}]
		\item $\xi_0$ is said to be of type 1 if there is a $T>0$ so that $\theta_{\xi_0}(T)=\theta^*$ and $\xi_{\xi_0}'(t)\ne 0$ for all $t\in (0,T)$. See figure \ref{fig3.1}\subref{fig:3.1a} for the schematic.
		\item $\xi_0$ is said to be of type 2 if there is a $T>0$ so that $\xi_{\xi_0}'(T)= 0$ and $\theta_{\xi_0}(t)\ne\theta^*$ for all $t\in (0,T)$. See figure \ref{fig3.1}\subref{fig:3.1b} for the schematic.
		\item $\xi_0$ is said to be of type 3 if $\xi_{\xi_0}'(t) \ne 0$ and $\theta_{\xi_0}(t)\ne\theta^*$ for all $t>0$. See figure \ref{fig3.1}\subref{fig:3.1c} for the schematic.
	\end{enumerate}
\end{definition}
\noindent
	\begin{figure}[H]\centering
		\subfigure[] {	
			\includegraphics[height=3cm]{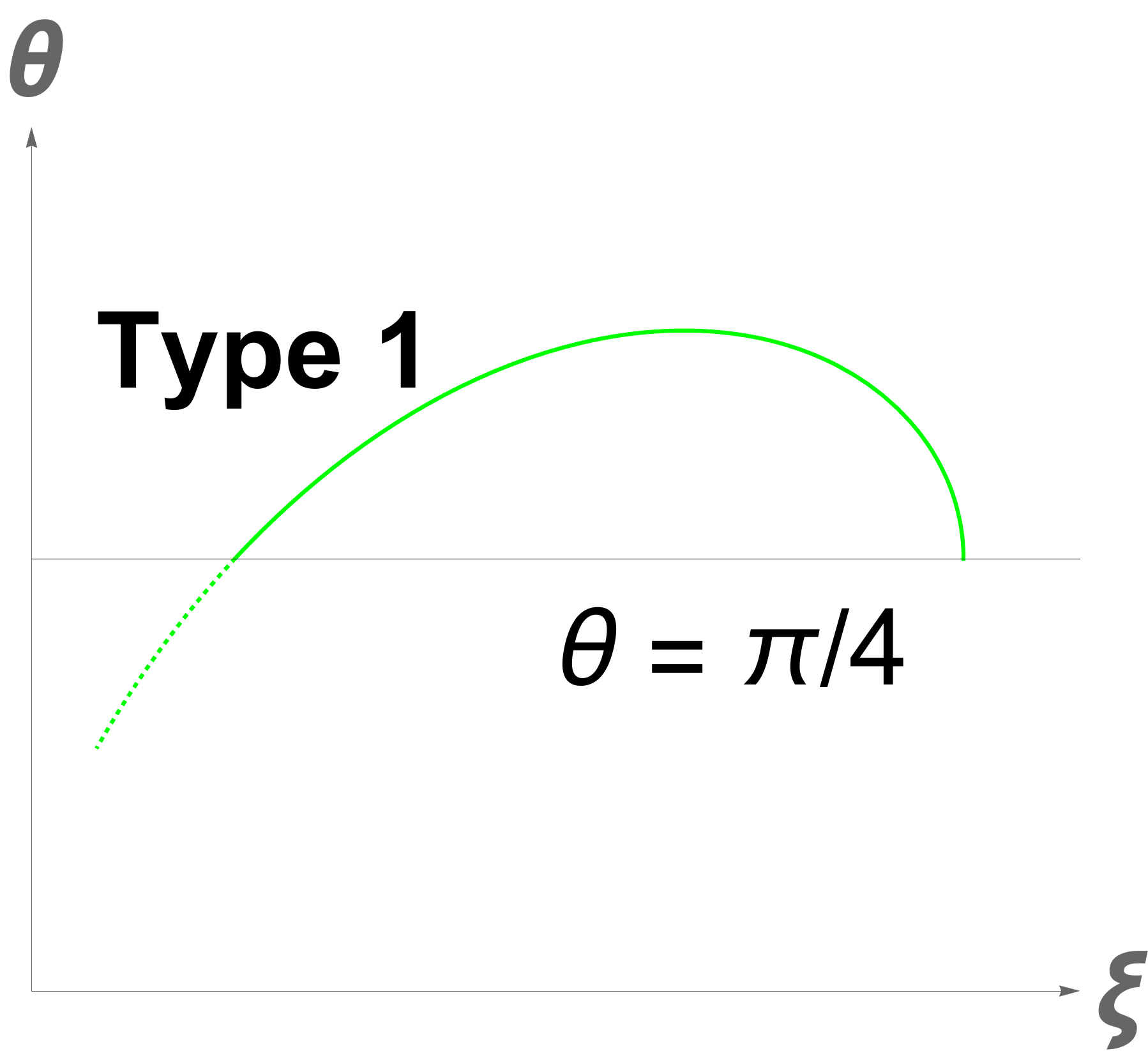}  \label{fig:3.1a}
		}\hfil
		\subfigure[] {
			\includegraphics[height=3cm]{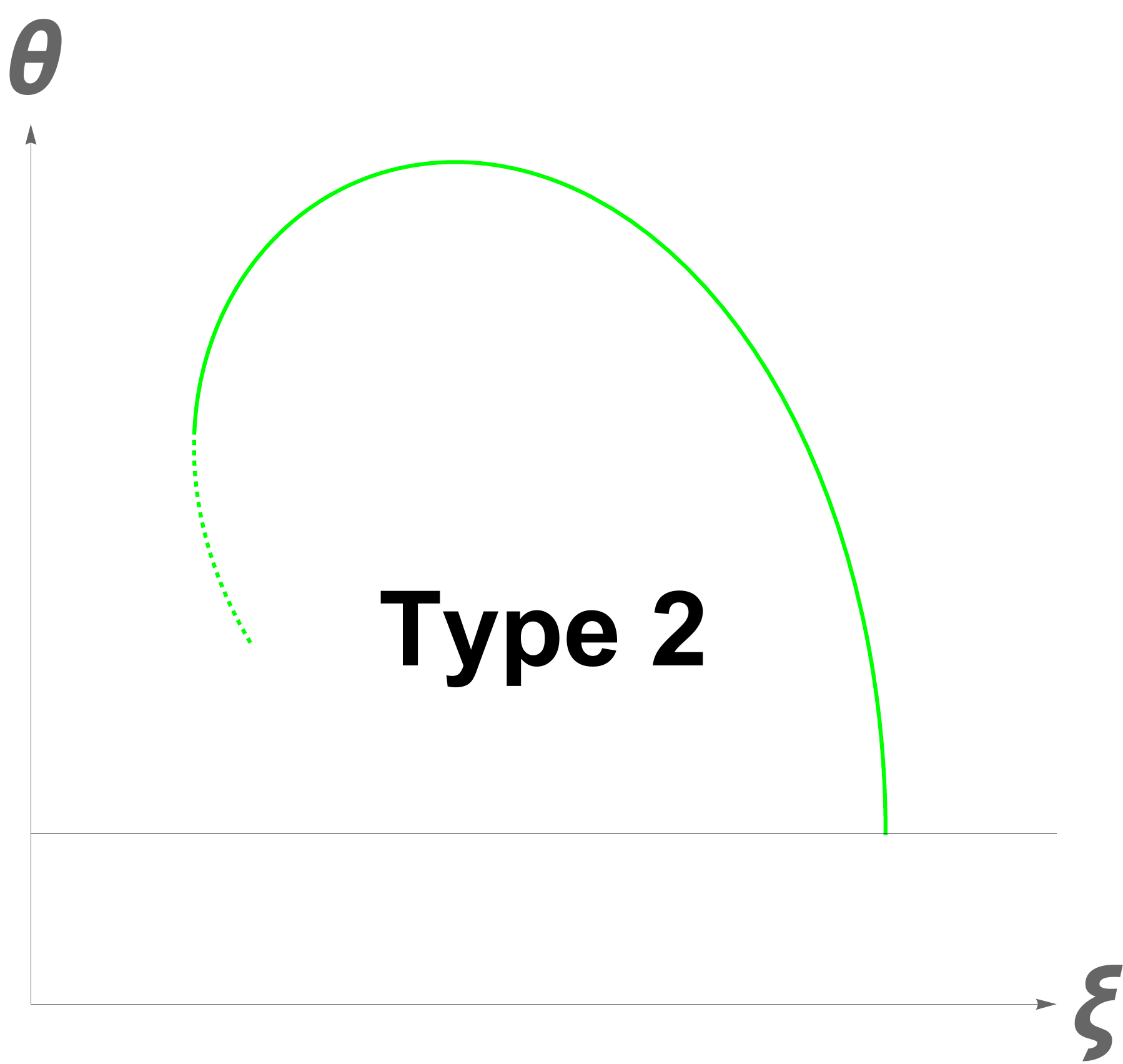}   \label{fig:3.1b}
		} \hfil
		\subfigure[] {
			\includegraphics[height=3cm]{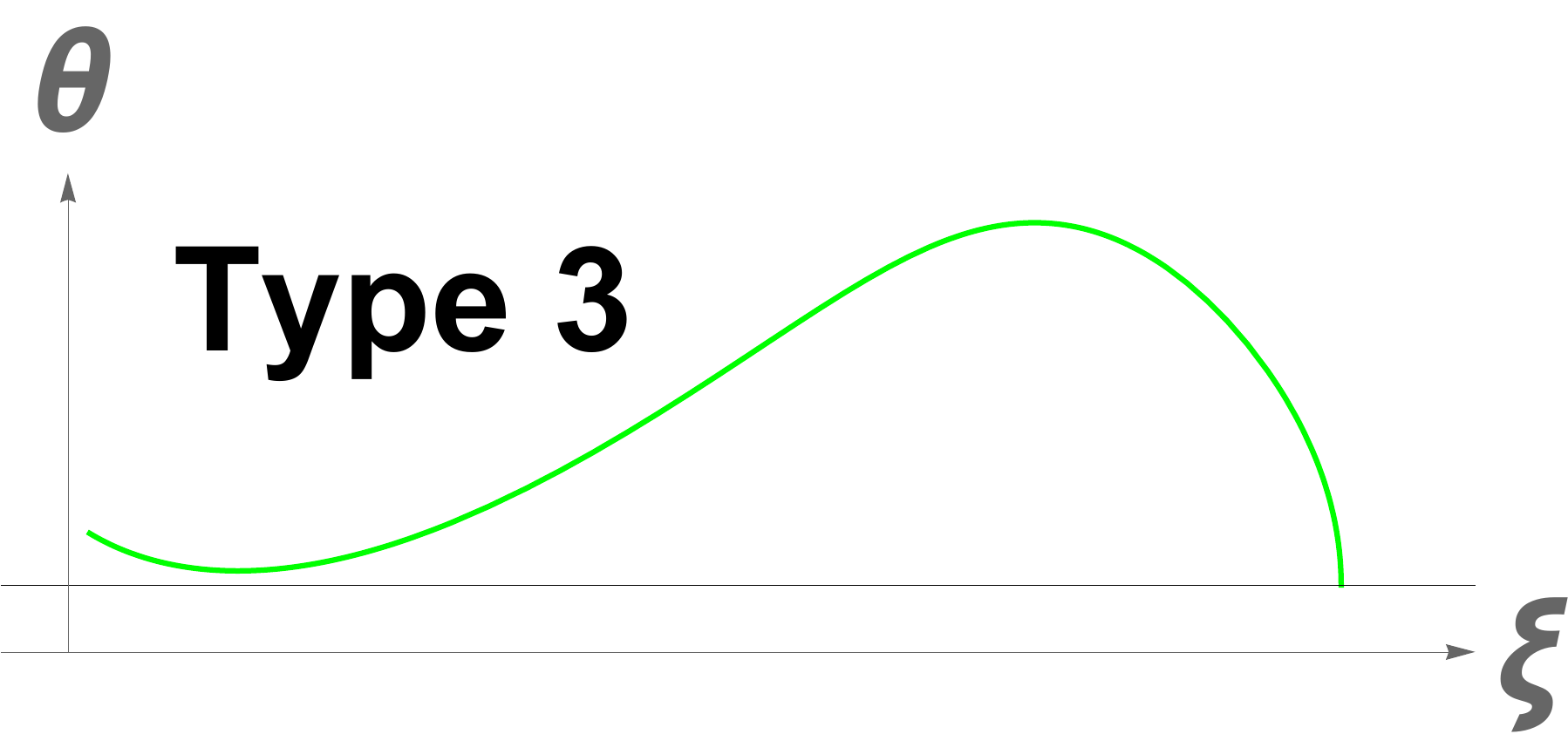}     \label{fig:3.1c}
		}
		\caption{\ Types of the profile curves.}
		\label{fig3.1}
	\end{figure}
Note that these three situations cover all possibilities.
Let $S_\lambda:=\frac{g}{2}\ln\left( \sqrt{\frac{2}{g}}\frac{-\lambda+\sqrt{\lambda^2+4n}}{2}\right) $, $S_{-\lambda}:=\frac{g}{2}\ln\left( \sqrt{\frac{2}{g}}\frac{\lambda+\sqrt{\lambda^2+4n}}{2}\right) $ and define
\[
\xi_0^* :=\inf\left\lbrace \xi\in \R\,|\,\xi_0\ \mathrm{is\ type\ 1\ for\ all}\ \xi_0>\xi\right\rbrace .
\]
\begin{proposition}\label{prop:10-25-5}
	Let $\lambda\ge0$, then we have:
	\begin{enumerate}[{\rm (i)}]
		\item $\xi_0^* <\infty$.
		\item $\xi_0^* >S_\lambda$.
		\item $\xi_0^*$ is not type 3.
	\end{enumerate}
\end{proposition} 
This proposition will be proved in section \ref{sec:4}.
To proceed, we give the following lemma, the proof is trivial so we omit it.
\begin{lemma}\label{lem:10-25-6}
	For $(\xi, \theta, \alpha)(t)\in D$ one has:
	\begin{enumerate}[{\rm (i)}]
		\item If $\alpha(t)\in \frac{\pi}{2}+2\pi\,\Z$, then $\alpha'(t)<0$ if and only if $\xi(t)<S_\lambda$; $\alpha'(t)>0$ if and only if $\xi(t)>S_\lambda$; $\alpha'(t)=0$ if and only if $\xi=S_\lambda$ for all time.
		\item If $\alpha(t)\in -\frac{\pi}{2}+2\pi\,\Z$, then $\alpha'(t)>0$ if and only if $\xi(t)<S_{-\lambda}$; $\alpha'(t)<0$ if and only if $\xi(t)>S_{-\lambda}$; $\alpha'(t)=0$ if and only if $\xi=S_{-\lambda}$ for all time.
		\item If $\lambda \ge 0$, $\alpha(t)\in 2\pi\,\Z$ and $\theta(t)<\theta^*$ then $\alpha'(t)>0$.
		\item If $\lambda \ge 0$, $\alpha(t)\in\pi+ 2\pi\,\Z$ and $\theta(t)>\theta^*$ then $\alpha'(t)>0$.
	\end{enumerate}
\end{lemma} 
This now gives:
\begin{proposition}\label{prop:10-25-7}
	Let $\lambda\ge 0$ then $\xi_0^*$ is type 1 and type 2.
\end{proposition}
\begin{proof}
	Let us begin by noting that \(\alpha_{\xi_0^*}'(0)>0\) since \(\xi_0^*>S_\lambda\) by proposition \ref{prop:10-25-5}(ii). 
	Together with \(\alpha_{\xi_0^*}(0)=\frac{\pi}{2}\), one gets for all \(\varepsilon>0\) small enough 
	\[
	\alpha'_{\xi_0^*}(t)>0, \ \ 0<\alpha_{\xi_0^*}(t)<\pi\ \ \  \forall t \in [0,\varepsilon].
	\]

	Since \(\xi_0^*\) is not type 3 by proposition \ref{prop:10-25-5}(iii), it must be at least one of type 1 or type 2.
	We first assume that \(\xi_0^*\) is type 1 but not type 2, 
	then we have a \(T>0\) so that for all \(\varepsilon>0\) small enough, one gets:
	\[
	\cos\alpha_{\xi_0^*}(t)\ne 0\ \ \ \forall t \in [\varepsilon,T+\varepsilon], \ \ \ \ \ \ \theta_{\xi_0^*}(T+\varepsilon)< \theta^*.
	\]

	Since the solutions of \eqref{eq:10-25-0} vary continuously in the initial conditions with respect to the topology of uniform convergence on compacta, one finds a neighbourhood \(U(\varepsilon)\) of \(\xi_0^*\) so that for all \(\xi_0 \in U(\varepsilon)\), one has 
	\[
		\alpha'_{\xi_0}(t)>0, \ \ 0<\alpha_{\xi_0}(t)<\pi\ \ \  \forall t \in [0,\varepsilon]
	\]
	and
	\[
	\cos\alpha_{\xi_0}(t)\ne 0\ \ \ \forall t \in [\varepsilon,T+\varepsilon], \ \ \ \ \ \ \theta_{\xi_0}(T+\varepsilon)< \theta^*.
	\]
	The above shows that for \(\xi_0\in U(\varepsilon)\), one has that \(\xi_0\) is type 1 which contradicting the definition of \(\xi_0^*\).

	The assumption that \(\xi_0^*\) is type 1 but not type 2 leads a contradiction via a similar argument, we carry out this:

	As before, we have a \(T>0\) so that for all \(\varepsilon>0\) small enough, one gets:
	\[
	\theta_{\xi_0^*}(t)>\theta^*\ \ \ \forall t \in[\varepsilon,T+\varepsilon].
	\]
	Without loss of generality, we may assume \(T>0\) is the first time satisfying definition \ref{def:10-25-4}(ii), hence \(\alpha_{\xi_0^*}(T)=\frac{\pi}{2}\) or \(\frac{3 \pi}{2}\).
	If \(\alpha_{\xi_0^*}(T)=\frac{\pi}{2}\) then \(\alpha'_{\xi_0^*}(T)\le0\), but lemma \ref{lem:10-25-6}(i) shows that \(\alpha'_{\xi_0^*}(T)\ne 0\).
	This implies that for all \(\varepsilon>0\) small enough one gets:
	\[
	\alpha_{\xi_0^*}(T+\varepsilon)<\frac{\pi}{2}.
	\]
	Then one gets a neighbourhood \(U(\varepsilon)\) of \(\xi_0^*\) so that for all \(\xi_0 \in U(\varepsilon)\), one has 
	\[
		\alpha'_{\xi_0}(t)>0, \ \ 0<\alpha_{\xi_0}(t)<\pi\ \ \  \forall t \in [0,\varepsilon]
	\] 
	and
	\[
		\theta_{\xi_0}(t)>\theta^*\ \ \ \forall t \in[\varepsilon,T+\varepsilon],\ \ \ \ \ 	\alpha_{\xi_0}(T+\varepsilon)<\frac{\pi}{2}.
	\]
	The above shows that for \(\xi_0 \in U(\varepsilon)\), one has that \(\xi_0\) is not type 1 which again contradicting the definition of \(\xi_0^*\).

	The other case can be treated similarly.
\end{proof}
Now theorem \ref{thm:2-27-1} follows from proposition \ref{prop:2-27-1} and proposition \ref{prop:10-25-7}.
Figure \ref{fig3.2} shows some periodic curves with \(\lambda=1\).
Note that the only possible value of \(g\) are 1, 2, 3, 4, 6. 
It is one of the primary results in M\"{u}nzner's seminal works \cite{Mu1,Mu2}, see also theorem~3.49 on page~136 of \cite{Ce}.
If \(g=3\) then \(m_1=m_2\) necessarily, and the multiplicity equals 1, 2, 4 or 8, see Cartan's \cite{Ca}.
If \(g=4\) then the multiplicity equals 1 or 2 when \(m_1=m_2\), as classified by Grove and Halperin \cite{GH}. 
If \(g=6\) then \(m_1=m_2\) necessarily, and the multiplicity equals 1 or 2, see M\"{u}nzner's \cite{Mu2} and Abresch's \cite{Ab}.

\noindent
	\begin{figure}[ht]\centering
		\subfigure[\(g=3,m_1=1,\xi_0= 0.94822\)] {	
			\includegraphics[width=0.2\linewidth]{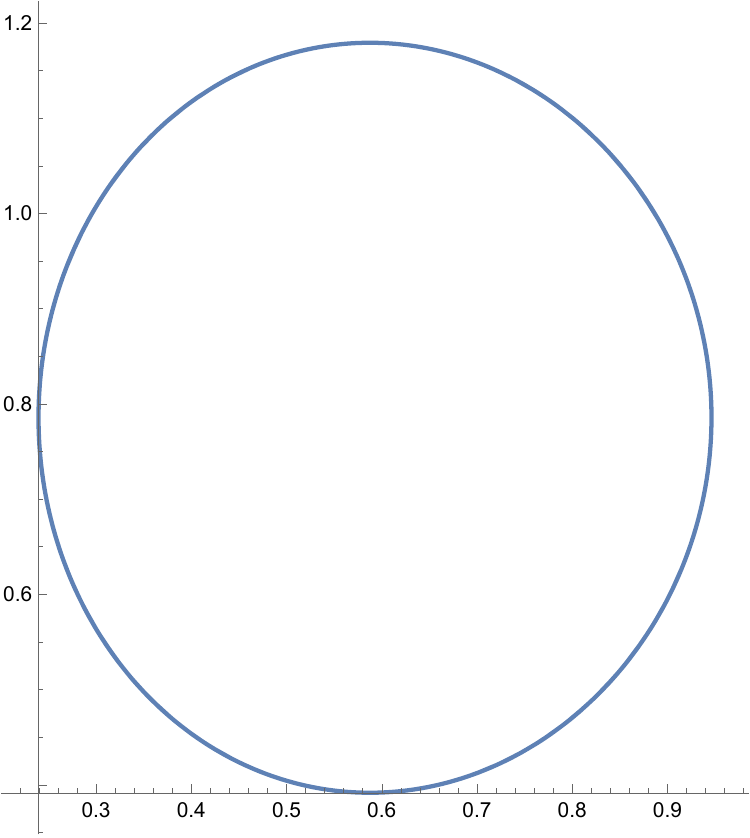}  \label{g=3,m1=1}
		}\hfil
		\subfigure[\(g=3,m_1=2,\xi_0= 1.3365\)] {
			\includegraphics[width=0.201\linewidth]{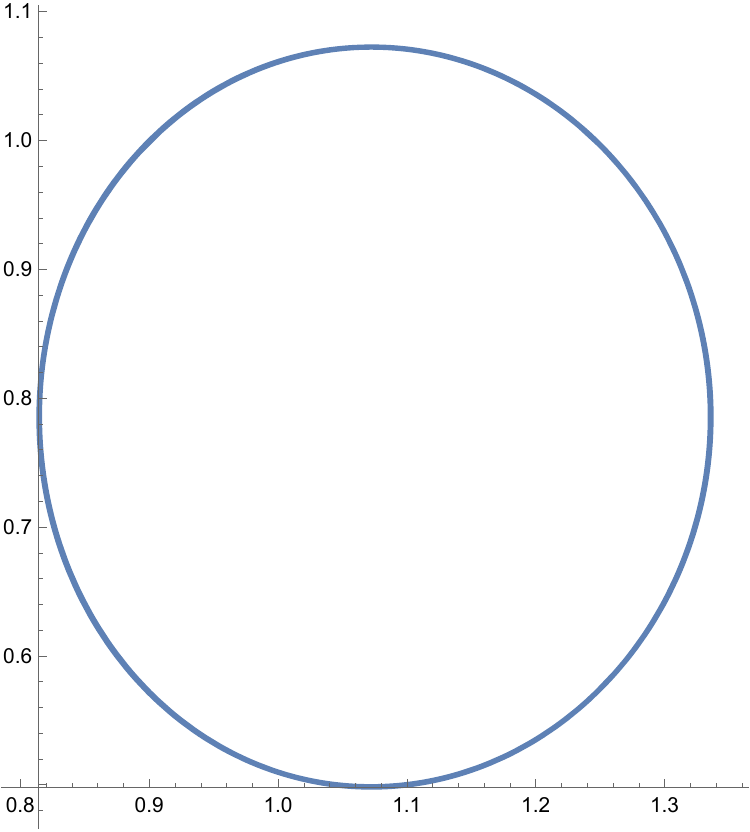}   \label{g=3,m1=2}
		}\hfil
		\subfigure[\(g=3,m_1=4,\xi_0= 1.76535\)] {
			\includegraphics[width=0.2\linewidth]{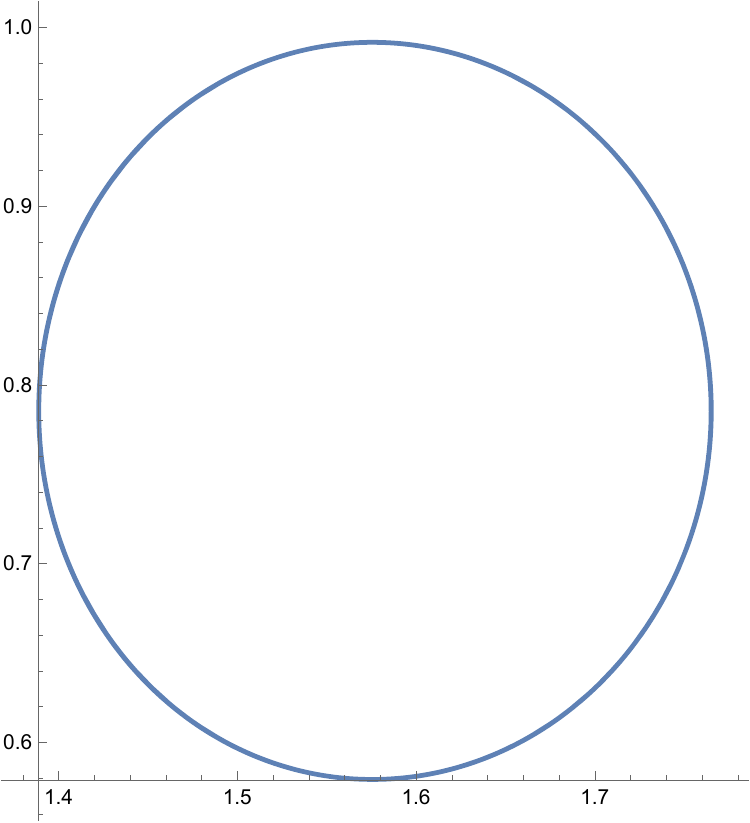}   \label{g=3,m1=4}
		}\hfil
		\subfigure[\(g=3,m_1=8,\xi_0= 2.22282\)] {
			\includegraphics[width=0.201\linewidth]{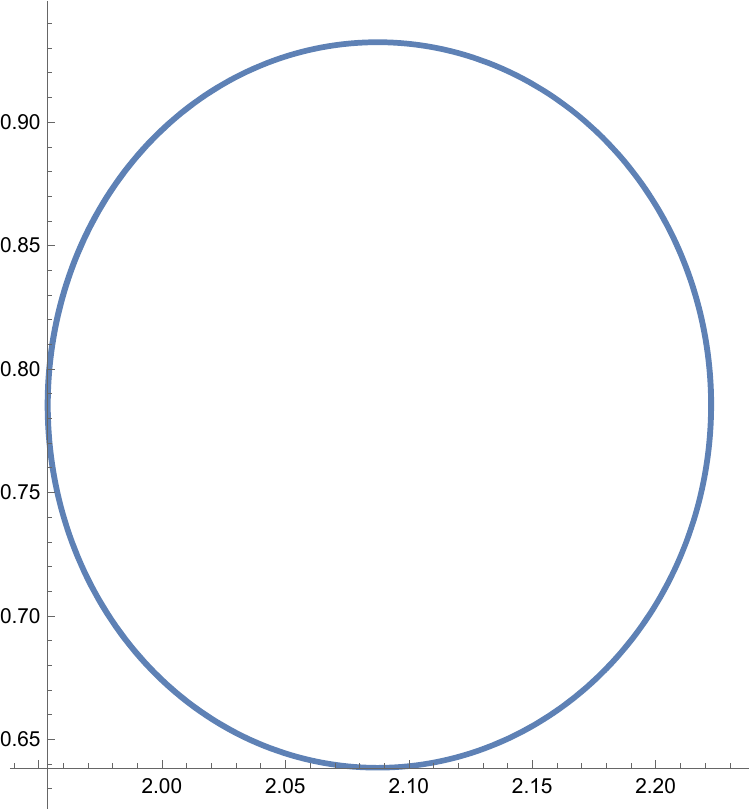}   \label{g=3,m1=8}
		}\hfil
		\subfigure[\(g=4,m_1=1,\xi_0= 1.15282\)] {	
			\includegraphics[width=0.2\linewidth]{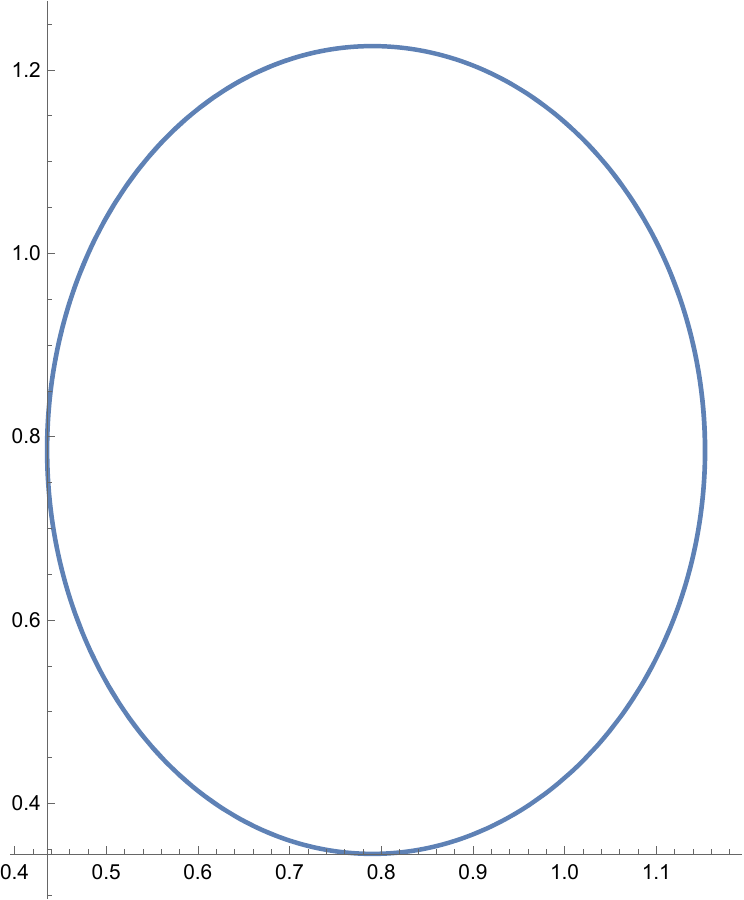}  \label{g=4,m1=1}
		}\hfil
		\subfigure[\(g=4,m_1=2,\xi_0= 1.70451\)] {
			\includegraphics[width=0.2\linewidth]{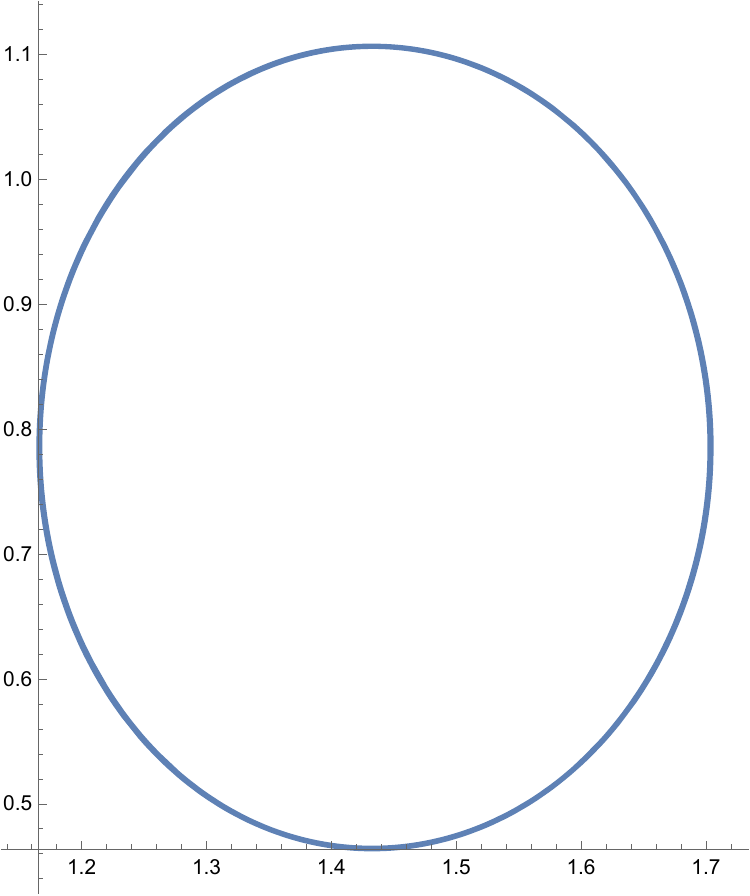}   \label{g=4,m1=2}
		}\hfil
		\subfigure[\(g=6,m_1=1,\xi_0= 1.53094\)] {	
			\includegraphics[width=0.2\linewidth]{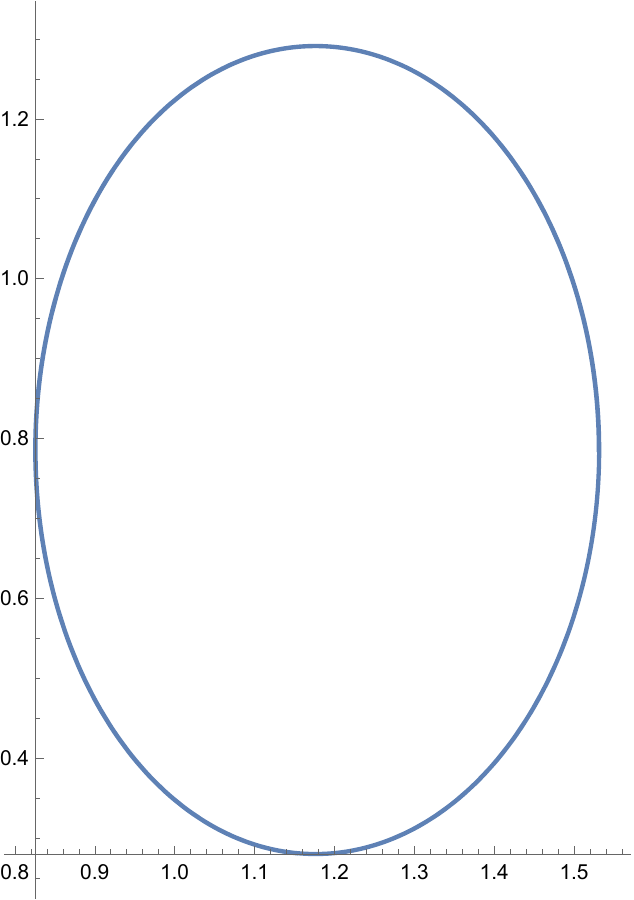}  \label{g=6,m1=1}
		}\hfil
		\subfigure[\(g=6,m_1=2,\xi_0= 2.4184\)] {
			\includegraphics[width=0.201\linewidth]{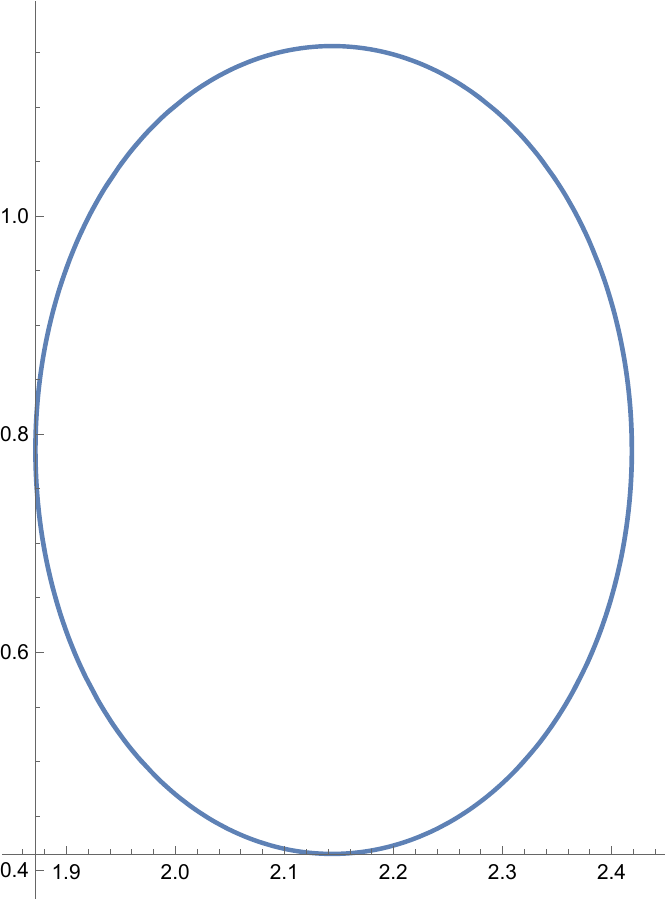}   \label{g=6,m1=2}
		}
		\caption{\ Some periodic curves with \(\lambda=1\) and \(m_1=m_2\).}
		\label{fig3.2}
	\end{figure}
%%%%%%%%%%%%%%%%%%%%%%%%%%%%%%%%%%%%%%%%%%%%%%%%%%%%%%%%%%%%%%%%%%%%%%%%%%%%%%%%%%%%%%%%%%%%%%%%%%%%%%%%%%%%%%%%%%%%%%%%%%%%%%%%%%%%%%%%%%%%%%%%%%%%%%%%%%%%%%%%%%%%%%%%%%%%%%%%%%

\section{Proof of Proposition \ref{prop:10-25-5}}\label{sec:4}
\subsection{Crossing in Finite Time}
\begin{lemma}\label{lem:10-26-1}
	Let $t_0$ be in $\R$, one has:
	\begin{enumerate}[{\rm (i)}]
		\item If $\xi(t_0)>S_\lambda$, $\xi'(t_0)<0$ and $\theta'(t)>0$ for all $t>t_0$, then there exists a time $T\in(0,\infty)$ so that $\xi(t_0+T)=S_\lambda$.
		\item If $\xi(t_0)>S_{-\lambda}$, $\xi'(t_0)<0$ and $\theta'(t)<0$ for all $t>t_0$, then there exists a time $T\in(0,\infty)$ so that $\xi(t_0+T)=S_{-\lambda}$.
		\item If $\xi(t_0)<S_\lambda$, $\xi'(t_0)>0$ and $\theta'(t)>0$ for all $t>t_0$, then there exists a time $T\in(0,\infty)$ so that $\xi(t_0+T)=S_\lambda$.
		\item If $\xi(t_0)<S_{-\lambda}$, $\xi'(t_0)>0$ and $\theta'(t)<0$ for all $t>t_0$, then there exists a time $T\in(0,\infty)$ so that $\xi(t_0+T)=S_{-\lambda}$.
	\end{enumerate}
\end{lemma}
\begin{proof}
    Only the first item will be proved, as the proofs for the others follow a similar approach.
	If this were not true then $\xi(t)>S_\lambda$ for all $t>t_0$. 
	By lemma \ref{lem:10-25-6}(i) one would then have that $\xi'(t)<0$ for $t\ge t_0$ since $\xi'(t_0)<0$.
%	, hence $\xi(t)$ converge. 
	Thus, with assumption, both $\theta(t)$ and $\xi(t)$ converge.
	Since $\xi(t)$ is bounded one finds from \eqref{eq:10-25-0} that $\alpha'(t)$ is bounded and then $\xi''(t)$, $\theta''(t)$ are also bounded.
	Combining with $\xi'<0$, $\theta'>0$ and the convergence of $\xi$ and $\theta$ one sees that $\xi'(t)=\cos\alpha\sin 2\theta$ and $\theta'(t)=\sin\alpha\sin 2\theta$ necessarily converge to 0, which imply that $\lim_{t \to \infty}\theta(t)=\frac{\pi}{2}$.
	Hence the trajectory $\left(\xi(t),\theta(t)\right)$, $t\in(t_0,\infty) $ can be viewed as a graph $(\xi(\theta), \theta) $ over $\theta $ in the interval $(\theta(t_0),\frac{\pi}{2}) $.
	As noted before, this graph satisfies \eqref{eq:10-25-1}.
	The limit of $\theta(t) $ together with the equation about $\alpha'$ in system \eqref{eq:10-25-0} give that $\lim_{t\to \infty}\alpha(t)=\frac{\pi}{2}$, which implies that $\lim_{\theta\to \frac{\pi}{2}^-}\frac{\mathrm{d}\xi}{\rm d\theta}	(\theta)= 0$. 
	Hence we get a solution of \eqref{eq:10-25-1} which is of class \(C^1((\theta(t_0),\frac{\pi}{2}])\cap C^2((\theta(t_0),\frac{\pi}{2}))\) with initial conditions \(\xi(\frac{\pi}{2})=\lim_{t\to\infty}\xi(t)\) and \(\frac{\mathrm{d}\xi}{\rm d\theta}(\frac{\pi}{2})=0\).
	By theorem 1.1 in \cite{Liang}, such solution is unique and is in fact of class \(C^2((\theta(t_0),\frac{\pi}{2}])\).
	Note that \(\frac{\mathrm{d}^2\xi}{\rm d\theta^2}(\frac{\pi}{2})\ge 0\) since \(\theta'(t)>0\) and \(\xi'(t)<0\) for all \(t>t_0\), and \eqref{eq:10-25-1} admits a constant solution \(\xi(\theta)=S_\lambda\).
	If \(\xi(\frac{\pi}{2})=S_\lambda\), we get two solutions with initial conditions \(\xi(\frac{\pi}{2})=S_\lambda\) and \(\frac{\mathrm{d}\xi}{\rm d\theta}(\frac{\pi}{2})=0\) which is a contradiction.
	If \(\xi(\frac{\pi}{2})>S_\lambda\), then one concludes from \eqref{eq:10-25-1} that  
	\[
		\frac{\mathrm{d}^2\xi}{\rm d\theta^2}\left(\frac{\pi}{2}\right)=-\frac{1}{1+m_2}\left(\e^{\frac{4}{g}\xi(\frac{\pi}{2})} -m+\lambda\sqrt{\frac{2}{g}}\,\e^{\frac{2}{g}\xi(\frac{\pi}{2})}\right)<0 
	\]
	which is also a contradiction. 
	This completes the proof.  
\end{proof}

\begin{lemma}\label{lem:10-26-2}
	If $\lambda\ge 0$ and for some $t_0$ one has $\theta(t_0)<\theta^*$ ($\theta(t_0)>\theta^*$) and $\theta'(t),\,\xi'(t)>0$ for all $t\ge t_0$ ($\theta'(t),\,\xi'(t)<0$ for all $t\ge t_0$) then there exists a time $T\in(0,\infty)$ so that $\theta(t_0+T)=\theta^*$.
\end{lemma}
\begin{proof}
	We only prove the first part of this lemma as the second part can be treated analogously.
	Let us first note that the assumptions yields that $\theta(t)$ converges.
	Next we divide the proof into two cases depending on whether \(\lim_{t\to \infty}\xi(t)\) exists.

	If $\lim_{t\to \infty}\xi(t)<\infty$, one then gets from \eqref{eq:10-25-0} that $\alpha'(t)$ remains bounded, which further implies the boundedness of $\xi''(t)$ and $\theta''(t)$.
	Together with $\xi'>0$, $\theta'>0$ and the convergence of $\theta$ and $\xi$ one gets that both $\theta'(t)=\sin\alpha\sin 2\theta$ and $\xi'(t)=\cos\alpha\sin 2\theta$  converge to 0.
	Hence $\lim_{t\to \infty}\sin 2\theta(t)=0$, namely, $\lim_{t\to \infty}\theta(t)=\frac{\pi}{2}$.
	In particular, there exists a time $T\in(0,\infty)$ so that $\theta(t_0+T)=\theta^*$.
	
	If $\lim_{t\to \infty}\xi(t)=\infty$, we proceed by a contradiction argument.
	Assuming $\lim_{t\to \infty}\theta(t)\le \theta^*$, one then concludes from the assumptions and \eqref{eq:10-25-0} that $\alpha'(t)>0$ for large $t$, which follows that $\theta'(t)=\sin\alpha\sin 2\theta>\varepsilon$ for large $t$ and some positive $\varepsilon$.
	We arrive at a contradiction by integrating $\theta'(t)$ over some interval, hence $\lim_{t\to \infty}\theta(t)> \theta^*$.
	This completes the proof of first part of this lemma.
%	namely, there exists a time $T\in(0,\infty)$ so that $\theta(t_0+T)=\theta^*$.
\end{proof}

%%%%%%%%%%%%%%%%%%%%%%%%%%%%%%%%%%%%%%%%%%%%%%%%%%%%%%%%%%%%%%%%%%%%%%%%%%%%%%%%%%%%%%%%%%%%%%%%%%%%%%%%%%%%%%%%%%%%%%%%%%%%%%%%%%%%%%%%%%%%%%%%%%%%%%%%%%%%%%%%%%%%%%%%%%%%%%%%%%
\subsection{Proof of Proposition \ref{prop:10-25-5}(i)} 
In what follows $\xi,\,\theta$ and $\alpha$ will denote the components of the solution of \eqref{eq:10-25-0} with initial condition $(\xi,\theta,\alpha)(0)=(\xi_0,\theta^*,\frac{\pi}{2})$.
Let us begin by noting that for $\xi_0>S_\lambda$ one has $\alpha'(0)>0$ whence $\alpha(t)\in (\frac{\pi}{2},\frac{\pi}{2}+\varepsilon)$, \(t \in(0,\varepsilon)\) for some small $\varepsilon>0$.
\begin{lemma}\label{lem:10-26-3}
	If $\xi_0$ is large enough then there is a time $T_1>0$ so that $\frac{\xi'(T_1)}{\theta'(T_1)}=-1$ while $\theta'(t),\,\xi'(t)<0$ for all $t\in (0,T_1]$.
\end{lemma}
\begin{proof}
	Let us treat it in four situations.
	If $\alpha(t_0)=\frac{\pi}{2}$ for some $t_0>0$, we may assume that $t_0$ is the first such time, 
%	at which $\alpha$ equals to $\frac{\pi}{2}$ 
	then $\alpha'(t_0) \le 0$ which implies $\xi(t_0)\le S_\lambda$ by lemma \ref{lem:10-25-6}(i);
	If $\alpha(t_0)=\frac{3\pi}{2}$ for some $t_0>0$, we may assume that $t_0$ is the first such time,
%	at which $\alpha$ equals to $\frac{3\pi}{2}$ 
	then $\alpha'(t_0) \ge 0$ which implies $\xi(t_0)\le S_{-\lambda}$ by lemma \ref{lem:10-25-6}(ii); If $\alpha(t)\in(\frac{\pi}{2},\frac{3\pi}{2})$ for all $t>0$ and there is a time $t_0>0$ at which $\alpha$ equals to $\pi$, then the lemma follows from differential mean value theorem; If $\alpha(t)\in(\frac{\pi}{2},\pi)$ for all $t>0$, then by lemma \ref{lem:10-26-1}\,{\rm (i)} there is a time $t_0\in (0,\infty)$ so that $\xi(t_0)=S_\lambda$.
	
	In any case except the third one, we have that $\xi(t_m)=\max\left\lbrace S_\lambda,S_{-\lambda} \right\rbrace$ for some $t_m>0$.
	Assume $t_m$ is the first such time, then $\alpha(t)\in (\frac{\pi}{2},\frac{3\pi}{2})$ for $t\in(0,t_m]$ by lemma \ref{lem:10-25-6}(i,\,ii).
	Therefore, for $\xi_0$ large enough there will then be some intermediate time $T_1<t_m$ for which $\frac{\theta'(T_1)}{\xi'(T_1)}=-1$ holds by Cauchy differential mean value theorem and intermediate value theorem.
    Let $T_1$ be the first such time then we have $\theta'(t),\,\xi'(t)<0$ for all $t\in (0,T_1]$.
\end{proof}
\begin{lemma}\label{lem:10-27-1}
	There are constants $c_1,\,c_2,\,c_3>0$ so that if $\xi_0$ is large enough one has
	\[
	c_1\,\e^{-\frac{4}{g}\xi_0}\le \theta(T_1)-\theta^*\le c_2\,\e^{-\frac{4}{g}\xi_0},\ \ \ \xi(T_1)\ge\xi_0-c_3\,\e^{-\frac{4}{g}\xi_0}.
	\]
\end{lemma}
\begin{proof}
	Let us first note that $\theta(T_1)\le\frac{\pi}{2}$ by the well definedness, whence by the differential mean value theorem $\xi_0-\xi(T_1)\le \frac{\pi}{2}-\theta^*$, namely, $\xi(T_1)\ge\xi_0-\frac{\pi}{2}+\theta^*$.
	For the proof of this lemma, it is more convenient to work with the following system of ODEs:
	\[
	\left\lbrace
	\begin{aligned}
	\xi'&= \cos\alpha, \\
	\theta'&= \sin\alpha, \\
	\alpha'&= \sin\alpha\,(\e^{\frac{4}{g}\xi} -m) + \cos\alpha\,(m_1\cot\theta - m_2\tan\theta)+\sqrt{\frac{2}{g}}\,\lambda\,\e^{\frac{2}{g}\xi}. 
	\end{aligned}
	\right.
	\]
	which one can get from \eqref{eq:10-25-0} by a reparametrization of time.
	We still use $T_1$ to denote the time at which $\frac{\xi'(T_1)}{\theta'(T_1)}=-1$ in this new ODE without confusion.
	One then gets for all $t\in[0,T_1]$:
	\[
	\cos\alpha(t)\in[-\frac{1}{\sqrt{2}},0],\ \ \ \sin\alpha(t)\in[\frac{1}{\sqrt{2}},1],
	\]
	which implies:
	\[
	\xi_0-\xi(T_1)\in[0,\frac{1}{\sqrt{2}}T_1],\ \ \ \theta(T_1)-\theta^*\in[\frac{1}{\sqrt{2}}T_1,T_1].
	\]
	It suffices to prove that $T_1$ has suitable upper and lower bounds.
	Noting that for $t\in[0,T_1]$ one has
	\[
	\alpha'(t)\ge\frac{1}{\sqrt{2}}\left(\e^{\frac{4}{g}\xi(T_1)} -m\right) -\sqrt{\frac{2}{g}}|\lambda|\,\e^{\frac{2}{g}\xi_0}
	\]
	and 
	\[
	\alpha'(t)\le\e^{\frac{4}{g}\xi_0}-\frac{\sqrt{2}}{2}(m_1\cot(\theta(T_1)) - m_2\tan(\theta(T_1)))+\sqrt{\frac{2}{g}}|\lambda|\,\e^{\frac{2}{g}\xi_0}.
	\]
	Integrating the first inequality from 0 to $T_1$ yields:
	\[
	\frac{\pi}{4}\ge T_1\left( \frac{1}{\sqrt{2}}\left(\e^{\frac{4}{g}\xi(T_1)} -m\right) -\sqrt{\frac{2}{g}}|\lambda|\,\e^{\frac{2}{g}\xi_0}\right). 
	\]
	Together with $\xi(T_1)\ge\xi_0-\frac{\pi}{2}+\theta^*$ one gets that $T_1\le d_1\,\e^{-\frac{4}{g}\xi_0}$ for an appropriate constant $d_1>0$, which implies that $\xi_0-\xi(T_1)\le c_3\,\e^{-\frac{4}{g}\xi_0}$ and $\theta(T_1)-\theta^*\le c_2\,\e^{-\frac{4}{g}\xi_0}$ for appropriate constants $c_2,\,c_3>0$.
	
	Combining $\theta(T_1)-\theta^*\le c_2\,\e^{-\frac{4}{g}\xi_0}$ with $m_1\cot(\theta^*) - m_2\tan(\theta^*)=0$ gives for $\xi_0$ large enough that $-m_1\cot(\theta(T_1)) + m_2\tan(\theta(T_1))<1$.
	Integrating the other inequality for $\alpha'(t)$ from 0 to $T_1$ then gives
	\[
	\frac{\pi}{4}\le T_1\left(\e^{\frac{4}{g}\xi_0}+\frac{\sqrt{2}}{2}+\sqrt{\frac{2}{g}}|\lambda|\,\e^{\frac{2}{g}\xi_0} \right) 
	\]
	which implies that $T_1\ge d_2\,\e^{-\frac{4}{g}\xi_0}$ for another appropriate constant $d_2>0$ provided $\xi_0$ is large enough.
	Now the final bound of the lemma follows and the proof is finished.
\end{proof}
\begin{lemma}\label{lem:10-27-2}
	Let $\lambda\ge 0$, then for $\xi_0$ large enough there is a time $T_2>T_1$ so that $\theta'(T_2)=0,\,\xi(T_2)>\xi_0-\frac{1}{\xi_0}$ while $\theta(t)\in\theta^*+ (0,\frac{1}{\xi_0})$ and $\xi'(t)<0$ for all $t\in(0,T_2]$.
\end{lemma}
\begin{proof}
	Let us begin by noting that $T_1<\frac{1}{4\xi_0}$ and  $\xi'(t)<0$ for $t\in (0,\frac{1}{\xi_0})$ provided $\xi_0$ is large enough, where the former follows from lemma \ref{lem:10-27-1} and the latter from lemma \ref{lem:10-25-6}(i,\,ii) since $\xi(t)-\xi_0=\int_{0}^{t}\sin 2\theta\cos\alpha\mathrm dt>-\frac{1}{\xi_0}$.
	Further as long as $\theta>\theta^*$ and $\alpha\in(\frac{\pi}{2},\pi)$, one has for $t\in (0,\frac{1}{\xi_0})$ and $\xi_0$ large enough that
	\[
	\alpha'= \sin 2\theta\sin\alpha\,(\e^{\frac{4}{g}\xi} -m) + 2\cos\alpha\,(m_1\cos^2\theta - m_2\sin^2\theta)+\sqrt{\frac{2}{g}}\,\lambda\,\sin 2\theta\,\e^{\frac{2}{g}\xi}
	\]
	is a sum of three positive(non-negative) terms and so $\alpha$ is increasing,
	where $\lambda\ge0$ were used.
	Recalling that by definition $\alpha(T_1)=\frac{3\pi}{4}$ and assuming $\alpha(t)<\pi$ for all $t\in(T_1,\frac{1}{\xi_0})$ yields
	\[
	\alpha'(t)>A\e^{\frac{4}{g}(\xi_0-\frac{1}{\xi_0})}(\pi-\alpha)
	\]
	where $A>0$ is some constant.
	For this estimate, we used that $\theta(t)$ is bounded away from $\left\lbrace 0,\frac{\pi}{2} \right\rbrace $ for $t\in (0,\frac{1}{\xi_0})$, which follows from $|\theta'(t)|\le 1$.
	From the intermediate value theorem, one then gets a $\tilde{t}\in(T_1,\frac{1}{2\xi_0})$ so that
	\[
	\alpha(\tilde{t})=\pi-\frac{\pi}{4}\exp\left(-A \e^{\frac{4}{g}(\xi_0-\frac{1}{\xi_0})}\frac{1}{4\xi_0}\right). 
	\]
	For $\xi_0$ large enough, one then finds $\alpha(\tilde{t})>\pi-\e^{-\xi_0^2}$.
	Now $\theta(T_1)>\theta^*+c_1\e^{-\frac{4}{g}\xi_0}$ from lemma \ref{lem:10-27-1}, so $l(\theta(T_1))\le \frac{1}{2}c_1l'(\theta^*)\e^{-\frac{4}{g}\xi_0}$ for $\xi_0$ large enough. 
	By assumption $\theta(\tilde{t})>\theta(T_1)$, so another estimate yields for $t\in (\tilde{t},\frac{1}{\xi_0})$:
	\[
	\alpha'(t)>B\,\e^{-\frac{4}{g}\xi_0}
	\]
	where $B>0$ is some constant incorporating the $\cos\alpha$ term (which is close to -1) and  $ c_1l'(\theta^*)$ (which is bounded away from 0). 
	This then yields
	\[
	\alpha(\tilde{t}+\frac{1}{2\xi_0})>\pi-\e^{-\xi_0^2}+B\,\e^{-\frac{4}{g}\xi_0}\frac{1}{2\xi_0}
	\]
	which is larger than $\pi$, contradicting our assumption that $\alpha(t)<\pi$ for all $t\in (T_1,\frac{1}{\xi_0})$. 
	Hence there is a time $T_2\in(T_1,\frac{1}{\xi_0}) $ so that $\alpha(T_2)=\pi$. 
	Let $T_2$ be the first such time then the proof completes. 
\end{proof}
\begin{lemma}\label{lem:10-27-3}
	Let $\lambda\ge 0$, then $\xi_0$ is of type 1 if $\xi_0$ is large enough.
\end{lemma}
\begin{proof}
	If $\xi_0$ is large enough by lemma \ref{lem:10-27-2}  there is a $T_2>0$ for which $\theta'(T_2)=0$.
	Note that $\alpha(T_2)=\pi,\,\theta(T_2)>\theta^*$ and then $\alpha'(T_2)>0$ by \eqref{eq:10-25-0} and $\lambda\ge0$.
	We first show that $\xi_0$ is not type 3 by contradiction.
	Suppose $\xi_0$ is of type 3, namely, $\alpha(t)\in (\frac{\pi}{2},\frac{3\pi}{2})$ and $\theta(t)\ne \theta^*$ for all $t>0$.
	Hence $\alpha(t)\in(\pi,\frac{3\pi}{2})$ for all $t>T_2$ by lemma \ref{lem:10-25-6}(iv), then it follows from lemma \ref{lem:10-26-2} that there is a time $T>T_2$ so that $\theta(T)=\theta^*$, which is a contradiction.
	
	We now assume $\xi_0$ is not type 1, so it must be type 2. Hence there is $T_3>T_2$ such that $\xi'(T_3)=0$, $\theta(t)>\theta^*$ and $\xi'(t)<0$ for all $t\in (0,T_3)$,
	which then follows from lemma \ref{lem:10-25-6}\,{\rm (iv)} that $\theta'(t)<0$ for $t\in (T_2,T_3)$.
	It follows that $\alpha(T_3)=\frac{3\pi}{2}$ and $\alpha'(T_3)\ge 0$, which implies $\xi(T_3)\le S_{-\lambda}$ by lemma \ref{lem:10-25-6}(ii).
	With $\xi(T_2)\ge \xi_0-\frac{1}{\xi_0}$ and $|\xi'(t)|\le 1$, one gets that $T_3-T_2$ will become arbitrarily large as $\xi_0$ grows.
	And so, for $\xi_0$ large enough, one sees that $T_3-1>T_2$ and $\xi(t)\le S_{-\lambda}+1$ for all $t\in [T_3-1,T_3]$.
	Hence, by \eqref{eq:10-25-0}, in this interval $\alpha'(t)$ and $\theta''(t)$ admit bounds independent of $\xi_0$.
	Since $\theta(T_3)<\theta(T_2)\le\theta^*+\frac{1}{\xi_0}$, one gets $\theta'(T_3)=-\sin 2\theta(T_3)\le -2\delta$ for some $\delta>0$ and $\xi_0$ large enough.
	Then the bound on $\theta''(t)$ gives a positive $b<1$ independent of $\xi_0$ so that $\theta'(t)<-\delta$ for all $t\in[T_3-b,T_3]$.
	
	But if $\xi_0$ is large enough, one notes that $\theta\in[\theta^*,\theta^*+\frac{1}{\xi_0}]$ and $\theta'(t)<-\delta$ cannot both simultaneously hold for all $t\in[T_3-b,T_3]$.
	This contradiction completes the proof. 
\end{proof} 
%%%%%%%%%%%%%%%%%%%%%%%%%%%%%%%%%%%%%%%%%%%%%%%%%%%%%%%%%%%%%%%%%%%%%%%%%%%%%%%%%%%%%%%%%%%%%%%%%%%%%%%%%%%%%%%%%%%%%%%%%%%%%%%%%%%%%%%%%%%%%%%%%%%%%%%%%%%%%%%%%%%%%%%%%%%%%%%%%%
\subsection{Proof of Proposition \ref{prop:10-25-5}(ii)} 
In what follows $\xi,\,\theta$ and $\alpha$ will denote the components of the solution of \eqref{eq:10-25-0} with initial condition $(\xi,\theta,\alpha)(0)=(S_\lambda+\varepsilon,\theta^*,\frac{\pi}{2})$ for positive $\varepsilon$,
and \(H(\theta)\) will denote \(\frac{1}{2}(m_1\cot\theta-m_2\tan\theta)\).
\begin{lemma}\label{lem:10-27-4}
	If $\xi_0=S_\lambda+\varepsilon$ is type 1, then there are $T_2(\varepsilon)>T_1(\varepsilon)>0$ so that $\theta'(T_2)=0$ and $\frac{\xi'(T_1)}{\theta'(T_1)}=-1$ while $\xi'(t)<0$ and $\theta'(t)>0$ for all $t\in(0,T_2)$.
\end{lemma}
\begin{proof}
	Assuming that $\xi_0$ is of type 1 means that there is a time $T>0$ for which $\theta(T)=\theta^*$ and $\alpha(t)\in(\frac{\pi}{2},\frac{3\pi}{2})$, $\theta(t)>\theta^*$ for all $t\in(0,T)$, then the lemma follows from differential mean value theorem.
\end{proof}
\begin{lemma}\label{lem:10-27-5}
	Let $\lambda\ge 0$, then for $\varepsilon$ small enough there is only one pair $(T_1,T_2)$ satisfying the conditions of lemma \ref{lem:10-27-4} and $\theta(T_1)\to \frac{\pi}{2}$ as $\varepsilon\to 0$.
\end{lemma}
\begin{proof}
	The time $T_2$ is obviously unique.
	
	On the other hand the initial condition $\varepsilon=0$ has as solution the line $(\xi,\theta,\alpha)=(S_\lambda,\arctan(\e^{2t}\tan\theta^*),\frac{\pi}{2})$.
	So one finds that as $\varepsilon\to 0$ the solution and its derivatives converge uniformly on compacta to the above curve,
	in particular for on any finite interval $[0,T]$, one can make $\xi'(t)/\theta'(t)$ arbitrarily small for all $t\in[0,T]$ by taking $\varepsilon$ small,
	while $\theta(T)$ is arbitrarily close to $\frac{\pi}{2}$ by taking \(T\) large and then $\varepsilon$ small.
	This means that as $\varepsilon \to 0$ one must have $\frac{\pi}{2}-\theta(T_1(\varepsilon))\to 0$, where $T_1$ is any of the times satisfying lemma \ref{lem:10-27-4}.
	
	Looking, however, at \eqref{eq:10-25-0} it is clear that $\alpha'(T_1)>0$ if $\theta(T_1)$ is close to $\frac{\pi}{2}$  while $\xi(T_1)$ is not too large.
	Hence, for $\varepsilon$ small enough one gets that for any pair $(T_1,T_2)$ satisfying \ref{lem:10-27-4} there is no pair $(T'_1,T_2)$ satisfying lemma \ref{lem:10-27-4} with $T'_1>T_1$.
    This completes the proof.
\end{proof}
Now we can introduce the following notations for brevity:
\[
\theta_1:=\theta(T_1),\ \ \ \theta_2:=\theta(T_2), \ \ \ \xi_2:=\xi(T_2).
\]
If \(\xi_0+\varepsilon\) is type 1, then there is a time $T>0$ for which $\theta(T)=\theta^*$ and $\alpha(t)\in(\frac{\pi}{2},\frac{3\pi}{2})$, $\theta(t)>\theta^*$ for all $t\in(0,T)$.
By lemma \ref{lem:10-27-4} one finds that \( T_2<T\) and the trajectory \(\{(\theta,\xi)(t)\,|\,t\in[0,T_2]\}\) is a graph of \(\xi\) over \(\theta\) in \([\theta^*,\theta_2]\), this graph satisfies \eqref{eq:10-25-1} and will be called the upper graph.
Further if in addition \(\lambda\ge0\) then the trajectory \(\{(\theta,\xi)(t)\,|\,t\in[T_2,T]\}\) is also a graph of \(\xi\) over \(\theta\) in \([\theta^*,\theta_2]\), this graph satisfies \eqref{eq:10-25-2} and will be called the lower graph.
In fact \(\alpha'(T_2)>0\) provided \(\lambda\ge0\) then one gets from lemma \ref{lem:10-25-6}(iv) that \(\alpha(t)\in(\pi,\frac{3\pi}{2})\) for \(t\in (T_2,T)\).
\begin{lemma}\label{lem:10-27-6}
	Let $\lambda\ge 0$ and $\xi_0=S_\lambda+\varepsilon$ be of type 1, then for $\varepsilon$ small enough one has that $\xi(T_1)\le S_\lambda$.
\end{lemma}
\begin{proof}
	We assume that $\xi(T_1)>S_\lambda$ and get a contradiction.
	The first step is to show this assumptions leads to constants $c_1,\,c_2>0$ independent of $\varepsilon$ such that for $\varepsilon$ small enough one has
	\begin{equation}\label{eq:1-12-3}
		c_1\varepsilon<\left( \frac{\pi}{2}-\theta_1\right) ^{m_2}<c_2\,\varepsilon. 
	\end{equation}
	By lemma \ref{lem:10-27-4} one concludes from  $\xi'(t)<0$ for all $t\in(0,T_1]$  that $\xi(t)>S_\lambda$ for such $t$.
	Since the upper graph satisfies the graph ODE \eqref{eq:10-25-1} and 
	$\lambda\ge0$ one then gets that $\frac{\mathrm{d}^2\xi}{\mathrm{d}\theta^2}<0$ for all $\theta\in(\theta^*,\theta_1)$ and that $\frac{\mathrm{d}\xi}{\mathrm{d}\theta}$ is strictly decreasing in this interval.
	The mean value theorem and the concavity of $\xi$ then implies 
	\[
	 \frac{\mathrm{d}\xi}{\mathrm{d}\theta}\left( \frac{\theta_1+\theta^*}{2}\right)  >-\varepsilon\,\frac{2}{\theta_1-\theta^*},\ \ \  \xi\left( \frac{\theta_1+\theta^*}{2}\right) >S_\lambda+\frac{\varepsilon}{2}. 
	\]
	Combining the second inequality with the graph ODE \eqref{eq:10-25-1} yields that \[\frac{\mathrm{d}^2\xi}{\mathrm{d}\theta^2}<-\hat c_1\varepsilon\] on $[\theta^*,\frac{\theta^*+\theta_1}{2}]$ for some $\hat c_1>0$.
	Using that $\theta_1\to\frac{\pi}{2}$ as $\varepsilon\to 0$ then gives constant $\tilde{c}_1,\ \tilde{c}_2$ such that 
	\begin{equation}\label{eq:1-12-1}
		\tilde{c}_1\varepsilon<\left| \frac{\mathrm{d}\xi}{\mathrm{d}\theta}\left( \frac{\theta_1+\theta^*}{2}\right) \right| <\tilde{c}_2\varepsilon.
	\end{equation} 
	Rewriting the graph ODE \eqref{eq:10-25-1} as 
	\begin{equation}\label{eq:1-12-2}
	\begin{aligned}
		\frac{\frac{\mathrm{d}^2\xi}{\mathrm{d}\theta^2}}{\frac{\mathrm{d}\xi}{\mathrm{d}\theta}\left( 1+\left(\frac{\mathrm{d}\xi}{\mathrm{d}\theta} \right)^2 \right) } =& -\frac{\e^{\frac{4}{g}\xi}-m+\lambda\sqrt{\frac{2}{g}}\e^{\frac{2}{g}\xi}}{\frac{\mathrm{d}\xi}{\mathrm{d}\theta}}+ \lambda\sqrt{\frac{2}{g}}\frac{1-\sqrt{1+\left( \frac{\mathrm{d}\xi}{\mathrm{d}\theta}\right)^2 }}{\frac{\mathrm{d}\xi}{\mathrm{d}\theta}}\e^{\frac{2}{g}\xi}\\
		&-m_1\cot\theta+m_2\tan\theta.
	\end{aligned}
	\end{equation}
	One notes that the first terms on the right-hand side of the above equation is $O(1)$ in the interval $\left(\frac{\theta_1+\theta^*}{2},\theta_1 \right) $ by \eqref{eq:1-12-1} and the monotonicity of $\frac{\mathrm{d}\xi}{\mathrm{d}\theta}$.
	Also the second and third terms are clearly \(O(1)\) in the same interval. 
	Integrating \eqref{eq:1-12-2} over this interval the gives
	\[
	\log\left|\frac{\mathrm{d}\xi}{\mathrm{d}\theta}\left(\frac{\theta_1+\theta^*}{2} \right)  \right| +O(1)= m_2\log(\cos\theta_1)+O(1).
	\]
	The approximation $\cos\theta=\frac{\pi}{2}-\theta+O\left( \left( \frac{\pi}{2}-\theta\right)^3 \right)$ then gives the bounds \eqref{eq:1-12-3} from \eqref{eq:1-12-1}.
	
	For the next step note that $H(\theta)-\left(\frac{\mathrm{d}\xi}{\mathrm{d}\theta}(\theta) \right) ^{-1}$ becomes $+\infty$ as $\theta\to\theta^*$ and $H(\theta^*)+1$ as $\theta\to\theta_1$, which for $\varepsilon$ small enough will be negative.
	Hence for small enough $\varepsilon$ there is a $\theta_0\in (\theta^*,\theta_1)$ so that $H(\theta_0)= \left(\frac{\mathrm{d}\xi}{\mathrm{d}\theta}(\theta_0) \right) ^{-1}$.
	In the same way as lemma \ref{lem:10-27-5} one shows that $\theta_0\to \frac{\pi}{2}$ as $\varepsilon\to 0$.
	
	We let \[ q:= \frac{\frac{\pi}{2}-\theta_0}{\frac{\pi}{2}-\theta_1}\] and show next that there is a constant $c_3$ such that $q^{m_2}\ge \frac{c_3}{\frac{\pi}{2}-\theta_0}$,
	and hence $q$ grows unboundedly as $\varepsilon\to0$.
	
	To show this we again integrate \eqref{eq:1-12-2}, this time from $\theta_0$ to $\theta\le\theta_1$.
	The result is 
	\[
	 \log\left|\frac{\mathrm{d}\xi}{\mathrm{d}\theta}(\theta_0) \right| -  \log\left|\frac{\mathrm{d}\xi}{\mathrm{d}\theta}(\theta) \right|+O(1)=m_2\log\left(\frac{\cos\theta}{\cos\theta_0} \right) +O(1).
	\]            
	Noting that $\frac{\rm d \xi}{\rm d\theta}(\theta_0)=\frac{1}{H(\theta_0)}=-\frac{2}{m_2}(\frac{\pi}{2}-\theta_0)+O\left( \left(\frac{\pi}{2}-\theta_0 \right)^2 \right) $ and again using the expansion of cosine close to $\frac{\pi}{2}$ implies the existence of $c_3,\ c_4>0$ so that:
	\begin{equation}\label{eq:1-13-1}
		 c_4\left| \frac{\rm d\xi}{\rm d\theta}(\theta)\right| \ge \left(\frac{\frac{\pi}{2}-\theta_0}{\frac{\pi}{2}-\theta} \right)^{m_2}\left(\frac{\pi}{2}-\theta_0 \right)\ge  c_3\left| \frac{\rm d\xi}{\rm d\theta}(\theta)\right|.
	\end{equation}
	Taking $\theta=\theta_1$ shows that $q^{m_2}\ge \frac{c_3}{\frac{\pi}{2}-\theta_0}$, in particular $q\to+\infty$ as $\varepsilon\to0$. 
	One the other hand, if one integrate $\frac{\rm d\xi}{\rm d\theta}(\theta)$ over $(\theta_{ 0},\theta_1)$ one finds by \eqref{eq:1-13-1} that 
	\[
	\begin{aligned}
	c_4\left| \xi(\theta_1)-\xi(\theta_0)\right| \ge 
	\begin{cases}
	\left( \frac{\pi}{2}-\theta_0\right) ^2\log q &  m_2=1, \\
	\frac{1}{m_2-1}\left( \frac{\pi}{2}-\theta_{\scriptstyle 0}\right) ^2 \left( q^{m_2-1}-1 \right) &  m_2>1.
	\end{cases}
	\end{aligned}
	\]
	By \eqref{eq:1-12-3} and $q^{m_2}\ge \frac{c_3}{\frac{\pi}{2}-\theta_0}$ we have 
	\[
	\left(\frac{\pi}{2}-\theta_0 \right)^2=\frac{\left( \frac{\pi}{2}-\theta_1 \right) ^{m_2}}{\left( \frac{\pi}{2}-\theta_0 \right) ^{m_2}}\, q^{m_2}\left( \frac{\pi}{2}-\theta_0 \right) ^2\ge c_1c_3\,\varepsilon\left( \frac{\pi}{2}-\theta_0 \right) ^{1-m_2}.
	\]
	Using unboundedness of $q$ one sees that $\left| \xi(\theta_1)-\xi(\theta_0) \right| $ will be larger than $\varepsilon$, contradicting our assumption that $\xi(\theta_1)\ge S_\lambda$.
\end{proof}
\begin{lemma}\label{lem:10-27-7}
	Let $\lambda\ge 0$ and $\xi_0=S_\lambda+\varepsilon$ be of type 1, then  
	\[
	\xi(T_2)\le S_\lambda-\left( 2^{\frac{1}{2m_2}}-1\right) \left( \frac{\pi}{2}-\theta_2\right) + O\left(\left(\frac{\pi}{2}-\theta_2 \right)^2  \right). 
	\]
	for \(\varepsilon\) small enough.
\end{lemma}
\begin{proof}
	By lemma \ref{lem:10-27-6} one has that $\xi(T_1)\le S_\lambda$.
	Since the upper graph satisfies \eqref{eq:10-25-1} and \(\lambda\ge0\) one gets that for $\varepsilon$ small enough $\frac{\rm d^2\xi}{\rm d\theta^2}(\theta)<0$ for all $\theta\in(\theta_1,\theta_2)$, 
	whence $\xi(T_2)\le\xi(T_1)-(\theta_2-\theta_1)$ and the statement reduces to checking $\theta_2-\theta_1=\left(2^{\frac{1}{2m_2}}-1 \right) \left(\frac{\pi}{2}-\theta_2 \right)+O\left(\left(\frac{\pi}{2}-\theta_2 \right)^2 \right)$.
	
	Again, as in the proof of lemma \ref{lem:10-27-6}, one gets by integrating the equation \eqref{eq:1-12-2} from $\theta_1$ to $\theta_2$ that
	\[
	\frac{1}{2}\log2=O\left(\frac{\pi}{2}-\theta_1 \right)-m_2\log\left(\frac{\cos\theta_2}{\cos\theta_1} \right).
	\]
	Note that $\cos\theta=\frac{\pi}{2}-\theta+O\left(\left(\frac{\pi}{2}-\theta \right)^3 \right)$ as $\theta$ becomes arbitrarily close to $\frac{\pi}{2}$.
	This together with the above equality give the existence of $c>0$ (independent of $\varepsilon$) so that $\frac{1}{c}<\frac{\frac{\pi}{2}-\theta_2}{\frac{\pi}{2}-\theta_1}<c$, 
	and give the following equality
	\[
	\frac{1}{2}\log2=O\left(\frac{\pi}{2}-\theta_1 \right)-m_2\log\left(\frac{\frac{\pi}{2}-\theta_2}{\frac{\pi}{2}-\theta_1}\right).
	\]
	Then the lemma follows from some basic arithmetic.
\end{proof}
\begin{lemma}\label{lem:10-27-8}
	If $\lambda\ge 0$ and $\varepsilon$ is small enough then $\xi_0= S_\lambda+\varepsilon$ is not type 1.
\end{lemma}
\begin{proof}
	We prove by contradiction argument.
	As note before the assumption $\xi_0=S_\lambda+\varepsilon$ is type 1 leads $(\xi,\theta)$ being the union of two graphs of $\xi$ over $\theta$.
	In the previous lemmas we investigated the upper graph which ends at the turning point $(\xi_2,\theta_2)$.
	
	For the lower graph, it is determined by the graph ODE \eqref{eq:10-25-2} and the initial conditions $\xi(\theta_2)=\xi_2,\ \lim_{\theta\to\theta_2^-}\frac{\rm d\xi}{\rm d\theta}(\theta)=+\infty$.
	The assumption that \(\lambda\ge0\) and $\xi_0$ is type 1 necessitates that for all $\theta\in (\theta^*,\theta_2)$ one has $\frac{\rm d\xi}{\rm d\theta}(\theta)>0$ for the lower graph.  
	Then one concludes from lemma \ref{lem:10-27-6} that 
	\begin{equation}\label{eq:1-13-2}
	\begin{aligned}
	\frac{\frac{\mathrm{d}^2\xi}{\mathrm{d}\theta^2}}{\frac{\mathrm{d}\xi}{\mathrm{d}\theta}\left( 1+\left(\frac{\mathrm{d}\xi}{\mathrm{d}\theta} \right)^2 \right) } =& -\frac{\e^{\frac{4}{g}\xi}-m}{\frac{\mathrm{d}\xi}{\mathrm{d}\theta}}+ \lambda\sqrt{\frac{2}{g}}\frac{\sqrt{1+\left( \frac{\mathrm{d}\xi}{\mathrm{d}\theta}\right)^2 }}{\frac{\mathrm{d}\xi}{\mathrm{d}\theta}}\e^{\frac{2}{g}\xi}
	-m_1\cot\theta+m_2\tan\theta\\
	>&-m_1\cot\theta+m_2\tan\theta.
	\end{aligned}
	\end{equation}
	Integrating form $\frac{\theta^*+\theta_2}{2}$ to $\theta_2$ one gets
	\begin{equation}\label{eq:1-14-2}
		\begin{aligned}
		-\log\frac{\frac{\rm d\xi}{\rm d\theta}\left( \frac{\theta^*+\theta_2}{2}\right)}{\sqrt{1+\frac{\rm d\xi}{\rm d\theta}\left( \frac{\theta^*+\theta_2}{2}\right)^2}}\ge&-m_2\log\cos\theta_2+O(1)\\
		\ge&-m_2\log\left(\frac{\pi}{2}-\theta_2 \right)+O(1).
		\end{aligned}
	\end{equation}
	This gives the existence of a $c>0$ so that 
	\begin{equation}\label{eq:1-14-1}
		\frac{\rm d\xi}{\rm d\theta}\left( \frac{\theta^*+\theta_2}{2}\right)\le \frac{c_1\left(\frac{\pi}{2}-\theta_2 \right)^{m_2}}{\sqrt{1-c_1^2\left(\frac{\pi}{2}-\theta_2 \right)^{2m_2}}}\le 2c_1\left(\frac{\pi}{2}-\theta_2 \right)^{m_2}.
	\end{equation}
	For $m_2>1$ this implies the lemma, since for all $\theta\in\left(\theta^*,\theta_2 \right)$ one has 
	\[
	\frac{\rm d^2\xi}{\rm d\theta^2}\left(\theta \right)>-\left.\left(\e^{\frac{4}{g}\xi}-m+\lambda\sqrt{\frac{2}{g}}\e^{\frac{2}{g}\xi} \right) \right|_{\xi=\xi(\theta_2)}\ge c_2\left(\frac{\pi}{2}-\theta_2 \right)
	\]
	 with some constant by lemma \ref{lem:10-27-7}, and then 
	 \[
	 \frac{\rm d\xi}{\rm d\theta}\left(\theta \right)\le2c_1\left(\frac{\pi}{2}-\theta_2 \right)^{m_2}-c_2\left(\frac{\pi}{2}-\theta_2 \right)\left( \frac{\theta^*+\theta_2}{2}-\theta \right)
	 \]
	 for $\theta\in\left(\theta^*, \frac{\theta^*+\theta_2}{2} \right)$ and one gets $\frac{\rm d\xi}{\rm d\theta}\left(\theta \right)=0$ for one such $\theta$.
	 
	 For $m_2=1$, one gets first from \eqref{eq:1-14-1} that $\frac{\rm d\xi}{\rm d\theta}$ takes all values in $\left(2c_1\left(\frac{\pi}{2}-\theta_2 \right),+\infty\right)$ as $\theta$ varies from $\frac{\theta^*+\theta_2}{2} $ to $\theta_2$.
	 In particular for $\theta_2$ close enough to $\frac{\pi}{2}$ there exists a $\theta_3\in\left( \frac{\theta^*+\theta_2}{2},\theta_2\right) $ so that $\frac{\rm d\xi}{\rm d\theta}\left(\theta_3 \right)=-\frac{1}{H(\theta_3)}$.
	 
	 Rewriting the graph ODE \eqref{eq:10-25-2} as 
	\begin{equation}\label{eq:1-28-1}
		\begin{aligned}
			\frac{\frac{\mathrm{d}^2\xi}{\mathrm{d}\theta^2}}{\frac{\mathrm{d}\xi}{\mathrm{d}\theta}\left( 1+\left(\frac{\mathrm{d}\xi}{\mathrm{d}\theta} \right)^2 \right) } =& -\frac{\e^{\frac{4}{g}\xi}-m-\lambda\sqrt{\frac{2}{g}}\e^{\frac{2}{g}\xi}}{\frac{\mathrm{d}\xi}{\mathrm{d}\theta}}+ \lambda\sqrt{\frac{2}{g}}\frac{\sqrt{1+\left( \frac{\mathrm{d}\xi}{\mathrm{d}\theta}\right)^2 }-1}{\frac{\mathrm{d}\xi}{\mathrm{d}\theta}}\e^{\frac{2}{g}\xi}\\
			&\ \ \ \ +(-m_1\cot\theta+m_2\tan\theta).
		\end{aligned}
	\end{equation}
	 By integrating \eqref{eq:1-28-1} one shows that $\theta_3$ gets arbitrarily close to $\frac{\pi}{2}$ if $\varepsilon$ is small enough.
	 We carry out this explicitly:
	 
	 Integrating \eqref{eq:1-28-1} from $\theta_3$ to $\theta_2$, the left hand side evaluates to $\frac{1}{2}\log\left(1+H(\theta_3)^2 \right)$,
	 which remains bounded as $\theta_2\to\frac{\pi}{2}$ unless $\theta_3$ also gets close to $\frac{\pi}{2}$. 
	 For the right hand side, one first note that there exists a $c_3>0$ so that \[-\frac{1}{H(\theta)}>c_3\left(\frac{\pi}{2}-\theta \right)\] 
	 for \(\theta\in(\theta^*,\frac{\pi}{2})\), and a \(M>0\) so that
	 \[\left|\e^{\frac{4}{g}\xi}-m-\lambda\sqrt{\frac{2}{g}}\e^{\frac{2}{g}\xi} \right| <M\]
	for \(\theta\in(\theta^*,\theta_2)\) 
	
	 Hence for $\theta\in\left(\theta_3,\theta_2 \right)$ one has 
	 \[
	 \begin{aligned}
	  \frac{\left|  \e^{\frac{4}{g}\xi}-m-\lambda\sqrt{\frac{2}{g}}\e^{\frac{2}{g}\xi}\right|}{\frac{\mathrm{d}\xi}{\mathrm{d}\theta}}<&\,\frac{M}{\frac{\rm d\xi}{\rm d\theta}(\theta_3)}=\frac{M}{\frac{-1}{H(\theta_3)}}<\frac{M}{c_3\left(\frac{\pi}{2}-\theta_3 \right)}
	 \end{aligned}
	 \]
	 and then the integral of the first term of the right hand side of \eqref{eq:1-28-1} from $\theta_3$ to $\theta_2$ is bounded.
	 The second term which is bounded yields also a bounded term.
	 The third term yields $-\log\frac{\frac{\pi}{2}-\theta_2}{\frac{\pi}{2}-\theta_3}+O(1)$ which is, crucially, unbounded unless $\theta_3\to\frac{\pi}{2}$ together with $\theta_2$.
	 So a situation where $\theta_2\to\frac{\pi}{2}$ but $\theta_3\not\to\frac{\pi}{2}$ is impossible.
	 
	In fact, this integral shows that  $q:=\frac{\frac{\pi}{2}-\theta_3}{\frac{\pi}{2}-\theta_2}$ grows unboundedly as $\theta_2$ approaches $\frac{\pi}{2}$.
	Finally if we integrate \eqref{eq:1-28-1} from $\theta\ge\theta_3$ to $\theta_2$ one gets
	 \[
	 -\log\left(\frac{\frac{\rm d\xi}{\rm d\theta}}{\sqrt{1+\left(\frac{\rm d\xi}{\rm d\theta} \right)^2}} \right)  =-\log\left(\frac{\frac{\pi}{2}-\theta_2}{\frac{\pi}{2}-\theta} \right)+O(1)
	 \]
	 which gives a constant $c_4>0$ so that \[\frac{\rm d\xi}{\rm d\theta}(\theta)\ge c_4\frac{\frac{\pi}{2}-\theta_2}{\frac{\pi}{2}-\theta}\] for $\theta\in[\theta_3,\theta_2]$.
	 Integrating this  from $\theta_3$ to $\theta_2$ gives \[\xi(\theta_2)-\xi(\theta_3)\ge c_4\left(\frac{\pi}{2}-\theta_2 \right)\log q,\]
	 then for \(\frac{\pi}{2}-\theta_2\) small enough one concludes from lemma \ref{lem:10-27-7} that 
	 \[
	 \xi(\theta_3)\le S_\lambda- \left( \frac{1}{2}\left( 2^{\frac{1}{2m_2}}-1\right)+c_4\log q\right) \left(\frac{\pi}{2}-\theta_2 \right). 
	 \]
	 Plugging this into \eqref{eq:10-25-2} then implies for any $\theta\in\left(\theta^*,\frac{\theta^*+\theta_2}{2} \right)$ that 
	 \[
	 \begin{aligned}
	 \frac{\rm d^2\xi}{\rm d\theta^2}\,\ge\, & m-\e^{\frac{4}{g}\xi(\theta_3)}-\lambda\sqrt{\frac{2}{g}}\,\e^{\frac{2\xi(\theta_3)}{g}} \\
	 \ge\,&m-\exp\left(\frac{4}{g} S_\lambda-\frac{4}{g}\left( \frac{1}{2}\left( 2^{\frac{1}{2m_2}}-1\right)+c_4\log q\right) \left(\frac{\pi}{2}-\theta_2 \right) \right) \\
	 &\ \ \ \ -\lambda\sqrt{\frac{2}{g}}\,\exp\left(\frac{2}{g} S_\lambda-\frac{2}{g}\left( \frac{1}{2}\left( 2^{\frac{1}{2m_2}}-1\right)+c_4\log q\right)\left(\frac{\pi}{2}-\theta_2 \right) \right)\\
	 \ge\,&\min\left\lbrace \frac{m}{2},c_5\left( \frac{1}{2}\left( 2^{\frac{1}{2m_2}}-1\right)+c_4\log q\right) \left(\frac{\pi}{2}-\theta_2 \right)\right\rbrace, 
	 \end{aligned}
	 \]
	 where \(c_5>0\) is a constant independent of \(\theta_2\).
	 By the unboundedness of $q$ as $\theta_2\to\frac{\pi}{2}$ and \eqref{eq:1-14-1}, one recovers $\frac{\rm d\xi}{\rm d\theta}(\theta)=0$ for some $\theta\in\left(\theta^*,\frac{\theta^*+\theta_2}{2} \right)$, provided $\frac{\pi}{2}-\theta_2$ is small enough.
	 
	 So also in the case $m_2=1$, we get a contradiction to the assumption that $S_\lambda+\varepsilon$ was type 1 for $\varepsilon$ small enough.
\end{proof}
%%%%%%%%%%%%%%%%%%%%%%%%%%%%%%%%%%%%%%%%%%%%%%%%%%%%%%%%%%%%%%%%%%%%%%%%%%%%%%%%%%%%%%%%%%%%%%%%%%%%%%%%%%%%%%%%%%%%%%%%%%%%%%%%%%%%%%%%%%%%%%%%%%%%%%%%%%%%%%%%%%%%%%%%%%%%%%%%%%
\subsection{Proof of Proposition \ref{prop:10-25-5}(iii)} 
\begin{lemma}\label{lem:10-27-9}
	Let $\lambda\ge 0$ and $\xi_0^*$ be of type 3, then:
	\begin{enumerate}[{\rm (i)}]
        \item  $\theta_{\xi_0}'(t)>0$ for all $t>0$.
		\item There is a $\delta>0$ and a $T>0$ so that $\e^{\frac{4}{g}\xi(t)} -m +\sqrt{\frac{2}{g}}\lambda\,\e^{\frac{2}{g}\xi(t)}<-\delta$ for all $t>T$.
		\item $\lim_{t\to \infty}\theta_{\xi_0^*}(t)=\frac{\pi}{2}$.
	\end{enumerate}
\end{lemma}
\begin{proof} 
	Let us first note that $\xi_{\xi_0^*}'(t)<0$ and $\theta_{\xi_0^*}'(t)>0$ for small $t>0$ since $\xi_0^*>S_\lambda$.
	Since $\xi_0^*$ is of type 3, then $\xi_{\xi_0^*}'(t)<0$ and $\theta_{\xi_0^*}(t)>\theta^*$ for all $t>0$.
	If $\theta_{\xi_0^*}'(t_0)=0$ for some $t_0>0$ then by lemma \ref{lem:10-25-6}(iv) one can show that $\theta_{\xi_0^*}'(t)<0$ for all $t>t_0$ by contradiction argument.
	Utilizing lemma \ref{lem:10-26-2} we get a $t_1>t_0$ so that $\theta_{\xi_0^*}(t_1)=\theta^*$ which is impossible.
	Hence part (i) follows.
	
	From part (i) we know that $\theta_{\xi_0^*}'(t)>0$ and $\xi_{\xi_0^*}'(t)<0$ for all $t<0$, 
	then by lemma \ref{lem:10-26-1}(i) $\xi_{\xi_0^*}(t)$ must reach $S_\lambda$ in finite time. 
	Hence part (ii) follows.
	
	Since $\theta_{\xi_0^*}$ is monotonous in $t$ and bounded by $\frac{\pi}{2}$,
	whence it converges. 
	Since $\xi_0^*$ is bounded above, one has that all the derivatives of the parameters are bounded and hence $\theta_{\xi_0^*}'(t)=\sin\alpha\sin2\theta\to0 $.
	If $\lim_{t\to+\infty}\theta_{\xi_0^*}(t)<\frac{\pi}{2}$ then $\lim_{t\to+\infty}\sin\alpha_{\xi_0^*}(t)=0$.
	By \eqref{eq:10-25-0}, this gives the existence of  $\varepsilon>0$ so that $\alpha'(t)>\varepsilon$ for large $t$ which is impossible.
	Hence part (iii) follows.
\end{proof}
\begin{lemma}\label{lem:10-27-10}
	If $\lambda\ge 0$, then $\xi_0^*$ is not type 3.
\end{lemma}
\begin{proof}
	Let $(\xi_{\varepsilon}(t),\theta_{\varepsilon}(t))$ denote the solution to \eqref{eq:10-25-0} with initial condition $(\xi,\theta,\alpha )(t=0)=(\xi_0^*+\varepsilon,\theta^*,\frac{\pi}{2})$ where $\varepsilon>0$.
	Note that $\xi_0^*+\varepsilon$ is type 1 by definition of $\xi_0^*$, in particular there is a $T_1(\varepsilon)$ so that $\theta_\varepsilon(t)$ has a local maximum and $\xi_\varepsilon'(t)<0$ for all $t\in(0,T_1]$.

	Assuming that $\xi_0^*$ is of type 3 and using $\xi_{\varepsilon}$ and $\theta_{\varepsilon}$ converge uniformly on compacta to $\xi_{\varepsilon=0}$ and $\theta_{\varepsilon=0}$ as $\varepsilon\to0 $ one finds that $\theta_{\varepsilon}(T_1)\to \frac{\pi}{2}$ as $\varepsilon\to0 $.
	This also implies $T_1(\varepsilon)\to \infty$ as $\varepsilon\to 0$,
	giving for $\varepsilon$ small enough that one has $\xi_{\varepsilon}(T_1)<S_\lambda-\delta$.
	
	After the extremum at $\theta_{\varepsilon}(T_1)$, one has that $\xi_{\varepsilon}$ becomes a lower graph over $\theta$.
	The graph ODE \eqref{eq:10-25-2} for $\xi $ yields: 
	\[
	\frac{\frac{\rm d^2\xi_{\varepsilon}}{\rm d\theta^2}}{\frac{\rm d\xi_{\varepsilon}}{\rm d\theta}\left(1+\left(\frac{\rm d\xi_{\varepsilon}}{\rm d\theta} \right)^2 \right)}=-\frac{\e^{\frac{4\xi_{\varepsilon}}{g}}-m}{\frac{\rm d\xi_{\varepsilon}}{\rm d\theta}}-(m_1\cot\theta-m_2\tan\theta)+
	\frac{\lambda\sqrt{\frac{2}{g}}\sqrt{1+\left(\frac{\rm d\xi_{\varepsilon}}{\rm d\theta} \right)^2}\e^{\frac{2\xi_{\varepsilon}}{g}}}{\frac{\rm d\xi_{\varepsilon}}{\rm d\theta}}.
	\]
	Integrating this from $x:=\frac{\theta^*+\frac{\pi}{2}}{2}$ to $\theta_\varepsilon(T_1)$ one gets \[-\log\frac{\frac{\rm d\xi_{\varepsilon}}{\rm d\theta}}{\sqrt{1+\left(\frac{\rm d\xi_{\varepsilon}}{\rm d\theta} \right)^2}}\left(x\right)\ge -m_2\log\left(\frac{\pi}{2}-\theta_{\varepsilon}(T_1) \right) +O(1)\]
	where $\lambda\ge0 $, $\frac{\rm d\xi_{\varepsilon}}{\rm d\theta  }>0$, $\e^{\frac{4\xi_{\varepsilon}}{g}}-m\le 0$ and $\frac{\rm d\xi_{\varepsilon}}{\rm d\theta}(\theta(T_1))=+\infty$ were used.
	
	As $\varepsilon\to 0$ this implies that $\frac{\rm d\xi_{\varepsilon}}{\rm d\theta}\left(x \right)$ becomes arbitrarily small, 
	in particular we may assume it to be smaller that $\frac{\delta(x-\theta^*)}{2}$.
	Remarking, however by \eqref{eq:10-25-2} and lemma \ref{lem:10-27-9} one gets \[\frac{\rm d^2\xi_\varepsilon}{\rm d\theta^2}(\theta)>\delta\] for all $\theta\in(\theta^*,x)$.
	With $\frac{\rm d\xi_\varepsilon}{\rm d\theta}(x)<\frac{\delta(x-\theta^*)}{2}$, one immediately gets $\frac{\rm d\xi_\varepsilon}{\rm d\theta}\left(\frac{\theta^*+x}{2} \right)<0$, contradicting that $\xi_0^*+\varepsilon$ is type 1. 
\end{proof}
%%%%%%%%%%%%%%%%%%%%%%%%%%%%%%%%%%%%%%%%%%%%%%%%%%%%%%%%%%%%%%%%%%%%%%%%%%%%%%%%%%%%%%%%%%%%%%%%%%%%%%%%%%%%%%%%%%%%%%%%%%%%%%%%%%%%%%%%%%%%%%%%%%%%%%%%%%%%%%%%%%%%%%%%%%%%%%%%%%

%%%%%%%%%%%%%%%%%%%%%%%%%%%%%%%%%%%%%%%%%%%%%%%%%%%%%%%%%%%%%%%%%%%%%%%%%%%%%%%%%%%%%%%%%%%%%%%%%%%%%%%%%%%%%%%%%%%%%%%%%%%%%%%%%%%%%%%%%%%%%%%%%%%%%%%%%%%%%%%%%%%%%%%%%%%%%%%%%%


\begin{thebibliography}{1}
		
		\bibitem{Ab} U. Abresch, \emph{Isoparametric hypersurfaces with four or six distinct principal curvatures. Necessary conditions on the multiplicities}, Math. Ann. {\bf 264} (1983), no.~3, 283--302; MR0714104
		\bibitem{A} S. B. Angenent, \emph{Shrinking doughnuts}, Birkh\"auser, Boston-Basel-Berlin,  {\bf 7}, 21-38, 1992.
		\bibitem{B} S. Brendle,  \emph{Embedded self-similar shrinkers of genus 0}, Ann. of Math. {\bf 183} (2016), 715-728.
		\bibitem{Ca} \'E. Cartan, \emph{Sur des familles remarquables d'hypersurfaces isoparam\'etriques dans les espaces sph\'eriques}, Math. Z. {\bf 45} (1939), 335--367; MR0000169
		\bibitem{Ce} T.~E. Cecil and P.~J. Ryan, {\it Geometry of hypersurfaces}, Springer Monographs in Mathematics, Springer, New York, 2015; MR3408101
		\bibitem{C} J. -E. Chang,  \emph{1-dimensional solutions of the $\lambda$-self shrinkers}, Geom. Dedicata  {\bf 189} (2017), 97-112.
		\bibitem{CM} T. H. Colding and W. P. Minicozzi II,  \emph{Generic mean curvature flow I; generic singularities}, Ann. of Math. {\bf 175} (2012), 755-833.
		\bibitem{CS} A. C. -P. Chu and A. Sun, \emph{Genus one singularities in the mean curvature flow}, arXiv: 2308.05923.
		\bibitem{CW}  Q. -M. Cheng and G. Wei,  \emph {Complete $\lambda$-hypersurfaces of weighted volume-preserving mean curvature flow},
		Cal. Var. Partial Differential Equations {\bf57} (2018), 32.
		\bibitem{CW1} Q. -M. Cheng and G. Wei, \emph{Examples of compact $\lambda$-hypersurfaces in Euclidean spaces},
		Sci. China Math.  {\bf64} (2021), 155-166.
		\bibitem{CLW} Q. -M. Cheng, J. Lai and G. Wei, \emph{Examples of compact embedded convex $\lambda$-hypersurfaces}, J. Funct. Anal. {\bf286} (2024), no. 2, Paper No. 110211.
        \bibitem{LCW} Q. -M. Cheng, J. Lai and G. Wei, \emph{Embedded cylindrical and doughnut-shaped $\lambda$-hypersurfaces}, arXiv:2406.11123.
		\bibitem{DD} M. do Carmo and M. Dajczer, \emph{Hypersurfaces in space of constant curvature},
		Trans. Amer. Math. Soc. {\bf 277}  (1983),  685-709.
		\bibitem{D} G. Drugan, \emph{An immersed $S^2$ self-shrinker}. Trans. Amer. Math. Soc. {\bf367} (2015), 3139-3159.
		\bibitem{DK} G. Drugan and S. J. Kleene, \emph{Immersed self-shrinkers}, Trans. Amer. Math. Soc. {\bf 369} (2017), 7213-7250.
		\bibitem{GH} K. Grove and S. Halperin, \emph{Dupin hypersurfaces, group actions and the double mapping cylinder}, J. Differential Geom. {\bf 26} (1987), no.~3, 429--459; MR0910016
		\bibitem{G} Q. Guang, \emph{A note on mean convex $\lambda$-surfaces in $\R^3$}, Proc. Amer. Math. Soc. {\bf 149} (2021), 1259-1266.

        \bibitem{H} G. Huisken, \emph{Asymptotic behavior for singularities of the mean curvature flow}, J. Differential Geom. {\bf 31} (1990), 285-299.


		\bibitem{Hei} S. Heilman, \emph{Symmetric convex sets with minimal Gaussian surface area}, Amer. J. Math. {\bf 143} (2021), 53-94.	

        \bibitem{I} T. Ilmanen, \emph{Problems in mean curvature flow}, available at \nolinkurl{http://people.math.ethz.ch/\textasciitilde ilmanen/classes/eil03/problems03.pdf}.
        

		\bibitem{KM} S. Kleene and N. M. M\o ller, \emph{Self-shrinkers with a rotational symmetry}, Trans. Amer. Math. Soc. {\bf 366} (2014), no. 8, 3943-3963.
		\bibitem{L} T. -K. Lee,  \emph{Convexity of $\lambda$-hypersurfaces}, Proc. Amer. Math. Soc. {\bf 150}  (2022), 1735-1744.
		
		\bibitem{Liang} J. Liang, \emph{A singular initial value problem and self-similar solutions of nonlinear dissipative wave equation}, J. Differential Equations {\bf 246}(2009), no.2, 819-844.

		\bibitem{LW} Z. Li and G. Wei, \emph{An immersed $S^n$ $\lambda$-hypersurface}, J. Geom. Anal. {\bf 33} (2023), Paper No. 288, 29 pp.
		\bibitem{M} P. McGrath, \emph{Closed mean curvature self-shrinking surfaces of generalized rotational type}, arXiv: 1507.00681.
		\bibitem{MR}  M. McGonagle and J. Ross, \emph{The hyperplane is the only stable, smooth solution to the isoperimetric problem in Gaussian space}, Geom. Dedicata {\bf178} (2015), 277-296.
		\bibitem{Mu1} H.~F. M\"unzner, \emph{Isoparametrische Hyperfl\"achen in Sph\"aren}, Math. Ann. {\bf 251} (1980), no.~1, 57--71; MR0583825
		\bibitem{Mu2} H.~F. M\"unzner, \emph{Isoparametrische Hyperfl\"achen in Sph\"aren. II. \"Uber die Zerlegung der Sph\"are in Ballb\"undel}, Math. Ann. {\bf 256} (1981), no.~2, 215--232; MR0620709
		\bibitem{R} J. Ross, \emph{On the existence of a closed, embedded, rotational $\lambda$-hypersurface}, J. Geom.  {\bf110} (2019), 1-12.
		\bibitem{R1} O. Riedler, \emph{Closed embedded self-shrinkers of mean curvature flow}, J. Geom. Anal. {\bf33} (2023), 172.
	    \bibitem{S} A. Sun, \emph{Compactness and rigidity of $\lambda$-surfaces}, Int. Math. Res. Not. IMRN  (2021),  11818-11844.
		\noindent
		
		
		
	\end{thebibliography}
\end{document}